\documentclass[11pt,reqno]{amsart}
\usepackage[a4paper,margin=28mm]{geometry}
\usepackage{amsmath,amssymb,amsthm,mathtools,booktabs,array,longtable,placeins}
\usepackage[T1]{fontenc}
\usepackage{lmodern,microtype,mathrsfs}
\usepackage{tikz}
\tikzset{
  dyn/.style={x=10.5mm,y=8.5mm,baseline=(current bounding box.center)},
  dn/.style={circle,draw,fill=white,inner sep=0pt,minimum size=4.4pt},
  rl/.style={font=\scriptsize,inner sep=1.2pt},
  frob/.style={<->,>=stealth,densely dotted,shorten >=1pt,shorten <=1pt},
  endpt/.style={fill=black!12,rounded corners=3pt},
}
\usepackage[colorlinks=true,linkcolor=blue,citecolor=blue,urlcolor=blue]{hyperref}
\hypersetup{pdftitle={Characters and Polynomiality for Maximal Parabolic Unipotent Radicals in Finite Classical Groups},pdfauthor={Qingchun Hao and Yang Yang}}
\numberwithin{equation}{section}
\newtheorem{theorem}{Theorem}[section]
\newtheorem{lemma}[theorem]{Lemma}
\newtheorem{proposition}[theorem]{Proposition}
\newtheorem{corollary}[theorem]{Corollary}
\newtheorem{conjecture}[theorem]{Conjecture}
\theoremstyle{definition}

\theoremstyle{remark}
\newtheorem{remark}[theorem]{Remark}
\newcommand{\BF}{\mathbb F}
\newcommand{\BN}{\mathbb N}
\DeclareMathOperator{\Mat}{Mat}
\DeclareMathOperator{\Sym}{Sym}
\DeclareMathOperator{\Alt}{Alt}
\DeclareMathOperator{\SHerm}{SHerm}

\DeclareMathOperator{\GL}{GL}
\DeclareMathOperator{\Sp}{Sp}
\DeclareMathOperator{\SO}{SO}
\DeclareMathOperator{\Ugrp}{U}
\DeclareMathOperator{\Irr}{Irr}
\DeclareMathOperator{\Hom}{Hom}
\DeclareMathOperator{\rank}{rank}
\DeclareMathOperator{\tr}{tr}
\DeclareMathOperator{\Tr}{Tr}
\DeclareMathOperator{\rad}{rad}

\DeclareMathOperator{\Ind}{Ind}

\DeclareMathOperator{\Lie}{Lie}
\newcommand{\gauss}[3]{\genfrac{[}{]}{0pt}{}{#1}{#2}_{#3}}
\title[Characters of maximal parabolic radicals]{Characters and Polynomiality for Maximal Parabolic Unipotent Radicals in Finite Classical Groups}
\author{Qingchun Hao}
\author{Yang Yang}
\address[Qingchun Hao and Yang Yang]{School of Mathematical Sciences, East China Normal University, Shanghai, China}
\email{52300155006@stu.ecnu.edu.cn}
\email{52275500011@stu.ecnu.edu.cn}
\keywords{Finite classical group, unipotent radical, irreducible character, Schubert variety, $R$-polynomial}
\begin{document}
\raggedbottom
\begin{abstract}
Let $q$ be a power of an odd prime. For the unipotent radicals of maximal rational parabolic subgroups of the finite symplectic, orthogonal and unitary groups, we determine all complex irreducible characters, their values, and the number of characters of each degree. Each of these numbers is a polynomial in $q-1$ with nonnegative integer coefficients, in analogy with a conjecture of Isaacs for unitriangular groups. For the orthogonal radicals, the character table and the power maps do not depend on the Witt type of the residual form, whereas the isomorphism type of the radical does when the isotropic dimension is at least three and the residual dimension is positive and even. Consequently, for $3\le a\le n-1$ the radicals of the stabilizers of totally isotropic $a$-spaces in $\SO^+_{2n}(q)$ and $\SO^-_{2n}(q)$ form Brauer pairs; for $a\le n-2$ both stabilizers are maximal parabolic subgroups. Finally, we express the multiplicities through $R$-polynomials, with the unequal parameters $q^2,\dots,q^2,q$ in the unitary case.
\end{abstract}
\maketitle

\section{Introduction}

Let $\BF_q$ be the field with $q$ elements, where $q$ is a power of an odd prime $p$, and let $\mathrm{UT}_n(q)$ be the group of upper unitriangular $n\times n$ matrices over $\BF_q$; it is a Sylow $p$-subgroup of $\GL_n(q)$. The irreducible representations of $\mathrm{UT}_n(q)$ are notoriously hard to classify, and already the number of its conjugacy classes is a difficult invariant. In 1960 Higman \cite{Higman} made the following conjecture.

\begin{conjecture}[Higman]
For every $n\ge1$, the number of conjugacy classes of $\mathrm{UT}_n(q)$ is a polynomial in $q$ with integer coefficients.
\end{conjecture}

Pak and Soffer \cite{PakSoffer} have verified this for $n\leq16$. Lehrer \cite{Lehrer} proposed an analogous statement for character degrees. Put
$$
 N_{n,e}(q)=\bigl|\{\chi\in\Irr(\mathrm{UT}_n(q)): \chi(1)=q^e\}\bigr|.
$$

\begin{conjecture}[Lehrer]
For all $n$ and $e$, the function $N_{n,e}(q)$ is a polynomial in $q$ with integer coefficients.
\end{conjecture}

Isaacs \cite{Isaacs} proposed a stronger form.

\begin{conjecture}[Isaacs]\label{conj:Isaacs}
For all $n$ and $e$, the function $N_{n,e}(q)$ is a polynomial in $q-1$ with nonnegative integer coefficients.
\end{conjecture}

In the same paper Isaacs determined $N_{n,e}(q)$ for the largest and the second largest character degree.

The same questions make sense for other families of unipotent groups. If $G(q)$ is a family of finite groups indexed by $q$, we say that it has the \emph{polynomial property} if for every $e$ the number of irreducible characters of $G(q)$ of degree $q^e$ is given by a polynomial in $q$ with integer coefficients, and that it has the \emph{$(q-1)$-positivity property} if these polynomials lie in $\BN[q-1]$. Nien \cite{Nien} constructed all irreducible representations of the unipotent radicals of three-block parabolic subgroups of $\GL_n(q)$ and proved $(q-1)$-positivity for them. On the side of conjugacy classes, Goodwin and R\"ohrle \cite{GoodwinRoehrle} showed that if a parabolic subgroup of $\GL_n(q)$ stabilizes a flag of length at most five, then the number of conjugacy classes of its unipotent radical is a polynomial in $q$ with integer coefficients, and R\"ohrle \cite{Roehrle} proved that a maximal parabolic subgroup of a classical algebraic group has only finitely many orbits on its unipotent radical. Sangroniz \cite{Sangroniz} studied the Sylow $p$-subgroups of the finite classical groups and their Lie algebras, in particular their characters of large degree. Supercharacter theories and orbit methods for these Sylow subgroups were developed by Andr\'e and Neto in types $\mathrm B$, $\mathrm C$ and $\mathrm D$ \cite{AndreNeto2006,AndreNeto2009,AndreNetoTheory2009}, by Andrews, whose work also covers the unitary groups \cite{Andrews}, and by Guo, Jedlitschky and Dipper \cite{GuoJedlitschkyDipper}. Those groups are in general much larger than the ones studied here, and the cited works construct supercharacters or orbit modules rather than the full set of irreducible characters.

In this paper we study the unipotent radicals of maximal rational parabolic subgroups of
$$
 \Sp_{2n}(q),\qquad \SO_{2n+1}(q),\qquad
 \SO^+_{2n}(q),\qquad \SO^-_{2n}(q),\qquad
 \Ugrp_N(q)\quad(N\ge2),
$$
where $n\ge1$, with $n\ge2$ for the even orthogonal groups, and $\Ugrp_N(q)$ is the isometry group of a nondegenerate Hermitian form on $\BF_{q^2}^N$. (For $n=1$ the group $\SO^\pm_2(q)$ is a torus and has no proper parabolic subgroups.) These radicals are not covered by the general linear results above, and, as we shall see, they are small enough for their characters to be computed completely.

\subsection{Setting}\label{sec:setting}
Each group $G$ in this list is the group $\mathbf G^F$ of fixed points of a Frobenius map $F$ on a connected reductive group $\mathbf G$ over $\overline{\BF}_q$. We fix an $F$-stable Borel subgroup $\mathbf B$ and an $F$-stable maximal torus $\mathbf T\subseteq\mathbf B$, and we number the set $\Delta$ of simple roots as in Figure~\ref{fig:dynkin}. The Frobenius induces a permutation $\rho$ of $\Delta$; it is trivial in the split cases and is shown by dotted arrows in the twisted cases. A standard parabolic subgroup is $F$-stable exactly when its set of simple roots is $\rho$-stable. Hence the maximal rational parabolic subgroups containing $\mathbf B^F$ are the groups $\mathbf P_{\mathcal O}^F$, where $\mathcal O$ is a $\rho$-orbit in $\Delta$ and $\mathbf P_{\mathcal O}$ is the standard parabolic subgroup whose Levi factor has simple roots $\Delta\setminus\mathcal O$; every maximal rational parabolic subgroup of $G$ is conjugate to exactly one of them.

\begin{figure}[!htb]
\centering
\small
\begin{tabular}{@{}c@{\hspace{12mm}}c@{}}
$\Sp_{2n}(q)$, type $\mathrm C_n$ & $\SO_{2n+1}(q)$, type $\mathrm B_n$\\[1mm]
\begin{tikzpicture}[dyn]
\fill[endpt] (3.72,-0.24) rectangle (4.28,0.24);
\node[dn,label={[rl]below:$\alpha_1$}] (1) at (0,0) {};
\node[dn,label={[rl]below:$\alpha_2$}] (2) at (1,0) {};
\node at (2,0) {$\cdots$};
\node[dn,label={[rl]below:$\alpha_{n-1}$}] (3) at (3,0) {};
\node[dn,label={[rl]below:$\alpha_n$}] (4) at (4,0) {};
\draw (1)--(2) (2)--(1.55,0) (2.45,0)--(3);
\draw[double distance=1.6pt] (3)--(4);
\draw (3.58,0.15)--(3.42,0)--(3.58,-0.15);
\end{tikzpicture}
&
\begin{tikzpicture}[dyn]
\fill[endpt] (3.72,-0.24) rectangle (4.28,0.24);
\node[dn,label={[rl]below:$\alpha_1$}] (1) at (0,0) {};
\node[dn,label={[rl]below:$\alpha_2$}] (2) at (1,0) {};
\node at (2,0) {$\cdots$};
\node[dn,label={[rl]below:$\alpha_{n-1}$}] (3) at (3,0) {};
\node[dn,label={[rl]below:$\alpha_n$}] (4) at (4,0) {};
\draw (1)--(2) (2)--(1.55,0) (2.45,0)--(3);
\draw[double distance=1.6pt] (3)--(4);
\draw (3.42,0.15)--(3.58,0)--(3.42,-0.15);
\end{tikzpicture}
\\[9mm]
$\SO^+_{2n}(q)$, type $\mathrm D_n$ & $\SO^-_{2n}(q)$, type ${}^2\mathrm D_n$\\[1mm]
\begin{tikzpicture}[dyn]
\fill[endpt] (3.62,0.42) rectangle (4.18,0.98);
\fill[endpt] (3.62,-0.98) rectangle (4.18,-0.42);
\node[dn,label={[rl]below:$\alpha_1$}] (1) at (0,0) {};
\node[dn,label={[rl]below:$\alpha_2$}] (2) at (1,0) {};
\node at (2,0) {$\cdots$};
\node[dn,label={[rl]below left:$\alpha_{n-2}$}] (3) at (3,0) {};
\node[dn,label={[rl]right:$\alpha_{n-1}$}] (4) at (3.9,0.7) {};
\node[dn,label={[rl]right:$\alpha_n$}] (5) at (3.9,-0.7) {};
\draw (1)--(2) (2)--(1.55,0) (2.45,0)--(3) (3)--(4) (3)--(5);
\end{tikzpicture}
&
\begin{tikzpicture}[dyn]
\fill[endpt] (3.62,-0.98) rectangle (4.5,0.98);
\node[dn,label={[rl]below:$\alpha_1$}] (1) at (0,0) {};
\node[dn,label={[rl]below:$\alpha_2$}] (2) at (1,0) {};
\node at (2,0) {$\cdots$};
\node[dn,label={[rl]below left:$\alpha_{n-2}$}] (3) at (3,0) {};
\node[dn,label={[rl]above:$\alpha_{n-1}$}] (4) at (3.9,0.7) {};
\node[dn,label={[rl]below:$\alpha_n$}] (5) at (3.9,-0.7) {};
\draw (1)--(2) (2)--(1.55,0) (2.45,0)--(3) (3)--(4) (3)--(5);
\draw[frob] (4) to[bend left=50] (5);
\end{tikzpicture}
\\[9mm]
$\Ugrp_{2n}(q)$, type ${}^2\mathrm A_{2n-1}$ & $\Ugrp_{2n+1}(q)$, type ${}^2\mathrm A_{2n}$\\[1mm]
\begin{tikzpicture}[dyn]
\fill[endpt] (3.52,-0.28) rectangle (4.08,0.28);
\node[dn,label={[rl]above:$\alpha_1$}] (t1) at (0,0.55) {};
\node[dn,label={[rl]above:$\alpha_2$}] (t2) at (1,0.55) {};
\node at (2,0.55) {$\cdots$};
\node[dn,label={[rl]above:$\alpha_{n-1}$}] (t3) at (3,0.55) {};
\node[dn,label={[rl]below:$\alpha_{2n-1}$}] (b1) at (0,-0.55) {};
\node[dn,label={[rl]below:$\alpha_{2n-2}$}] (b2) at (1,-0.55) {};
\node at (2,-0.55) {$\cdots$};
\node[dn,label={[rl]below:$\alpha_{n+1}$}] (b3) at (3,-0.55) {};
\node[dn,label={[rl]right:$\alpha_n$}] (m) at (3.8,0) {};
\draw (t1)--(t2) (t2)--(1.55,0.55) (2.45,0.55)--(t3) (t3)--(m);
\draw (b1)--(b2) (b2)--(1.55,-0.55) (2.45,-0.55)--(b3) (b3)--(m);
\draw[frob] (t1)--(b1);
\draw[frob] (t2)--(b2);
\draw[frob] (t3)--(b3);
\end{tikzpicture}
&
\begin{tikzpicture}[dyn]
\fill[endpt] (2.72,-0.86) rectangle (3.66,0.86);
\node[dn,label={[rl]above:$\alpha_1$}] (t1) at (0,0.55) {};
\node[dn,label={[rl]above:$\alpha_2$}] (t2) at (1,0.55) {};
\node at (2,0.55) {$\cdots$};
\node[dn,label={[rl]above:$\alpha_n$}] (t3) at (3,0.55) {};
\node[dn,label={[rl]below:$\alpha_{2n}$}] (b1) at (0,-0.55) {};
\node[dn,label={[rl]below:$\alpha_{2n-1}$}] (b2) at (1,-0.55) {};
\node at (2,-0.55) {$\cdots$};
\node[dn,label={[rl]below:$\alpha_{n+1}$}] (b3) at (3,-0.55) {};
\draw (t1)--(t2) (t2)--(1.55,0.55) (2.45,0.55)--(t3);
\draw (b1)--(b2) (b2)--(1.55,-0.55) (2.45,-0.55)--(b3);
\draw (t3)--(b3);
\draw[frob] (t1)--(b1);
\draw[frob] (t2)--(b2);
\draw[frob] (t3) to[bend left=50] (b3);
\end{tikzpicture}
\end{tabular}
\caption{Absolute Dynkin diagrams and the Frobenius action. The simple roots are $\alpha_i=e_i-e_{i+1}$ for $i<n$, together with $\alpha_n=2e_n$ in type $\mathrm C_n$, $\alpha_n=e_n$ in type $\mathrm B_n$, and $\alpha_{n-1}=e_{n-1}-e_n$, $\alpha_n=e_{n-1}+e_n$ in types $\mathrm D_n$ and ${}^2\mathrm D_n$; in the unitary cases $\alpha_i=e_i-e_{i+1}$ for $1\le i\le N-1$. Dotted arrows show the permutation $\rho$ induced by $F$, namely $\alpha_{n-1}\leftrightarrow\alpha_n$ in type ${}^2\mathrm D_n$ and $\alpha_i\leftrightarrow\alpha_{N-i}$ in type ${}^2\mathrm A_{N-1}$; in type ${}^2\mathrm A_{2n-1}$ the node $\alpha_n$ is fixed. The $\rho$-orbits are the nodes of the relative diagrams, of types $\mathrm B_{n-1}$, $\mathrm C_n$ and $\mathrm{BC}_n$ in the three twisted cases. The endpoint orbits are shaded. For the even orthogonal groups we assume $n\ge2$; the diagrams are drawn for $n\ge4$, and for $n=2,3$ the numbering is read in the obvious way.}\label{fig:dynkin}
\end{figure}
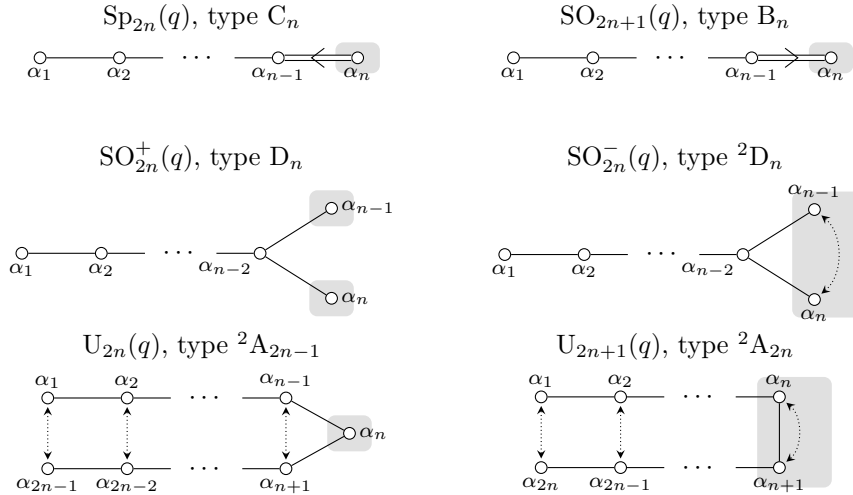

Each $\mathbf P_{\mathcal O}^F$ is the stabilizer in $G$ of a totally isotropic subspace $E$, defined over $k=\BF_q$ in the bilinear cases and over $K=\BF_{q^2}$ in the unitary case. We put
$$
 a=\dim E,\qquad M=\dim E^\perp/E,
$$
the dimensions being taken over $k$ or $K$, respectively. Table~\ref{tab:range} lists $a$ and $M$ for every orbit $\mathcal O$. For a given group the orbit is determined by $a$, with one exception: in type $\mathrm D_n$ the two orbits $\{\alpha_{n-1}\}$ and $\{\alpha_n\}$ both have $a=n$. They correspond to the two families of maximal totally isotropic subspaces and give non-conjugate parabolic subgroups, whose radicals are both isomorphic to $\Alt_n(k)$.

\begin{table}[!htb]
\centering
\caption{Maximal rational parabolic subgroups $\mathbf P_{\mathcal O}^F$, in the numbering of Figure~\ref{fig:dynkin}. Here $a=\dim E$, $M=\dim E^\perp/E$, and $\mathcal Z$ is the central summand of the radical. Endpoint orbits are marked by an asterisk.}\label{tab:range}
\begin{tabular}{@{}llllll@{}}
\toprule
Group & Type & Orbit $\mathcal O$ & $a$ & $M$ & $\mathcal Z$\\
\midrule
$\Sp_{2n}(q)$ & $\mathrm C_n$ & $\{\alpha_a\}$ & $1\le a\le n-1$ & $2(n-a)$ & $\Sym_a(k)$\\
 & & $\{\alpha_n\}^\ast$ & $n$ & $0$ & $\Sym_n(k)$\\
\addlinespace
$\SO_{2n+1}(q)$ & $\mathrm B_n$ & $\{\alpha_a\}$ & $1\le a\le n-1$ & $2(n-a)+1$ & $\Alt_a(k)$\\
 & & $\{\alpha_n\}^\ast$ & $n$ & $1$ & $\Alt_n(k)$\\
\addlinespace
$\SO^+_{2n}(q)$ & $\mathrm D_n$ & $\{\alpha_a\}$ & $1\le a\le n-2$ & $2(n-a)$ & $\Alt_a(k)$\\
 & & $\{\alpha_{n-1}\}^\ast$, $\{\alpha_n\}^\ast$ & $n$ & $0$ & $\Alt_n(k)$\\
\addlinespace
$\SO^-_{2n}(q)$ & ${}^2\mathrm D_n$ & $\{\alpha_a\}$ & $1\le a\le n-2$ & $2(n-a)$ & $\Alt_a(k)$\\
 & & $\{\alpha_{n-1},\alpha_n\}^\ast$ & $n-1$ & $2$ & $\Alt_{n-1}(k)$\\
\addlinespace
$\Ugrp_{2n}(q)$ & ${}^2\mathrm A_{2n-1}$ & $\{\alpha_a,\alpha_{2n-a}\}$ & $1\le a\le n-1$ & $2(n-a)$ & $\SHerm_a(K)$\\
 & & $\{\alpha_n\}^\ast$ & $n$ & $0$ & $\SHerm_n(K)$\\
\addlinespace
$\Ugrp_{2n+1}(q)$ & ${}^2\mathrm A_{2n}$ & $\{\alpha_a,\alpha_{2n+1-a}\}$ & $1\le a\le n-1$ & $2(n-a)+1$ & $\SHerm_a(K)$\\
 & & $\{\alpha_n,\alpha_{n+1}\}^\ast$ & $n$ & $1$ & $\SHerm_n(K)$\\
\bottomrule
\end{tabular}
\end{table}

By the \emph{endpoint} of a diagram we mean the $\rho$-orbit or orbits at its right-hand end, shaded in Figure~\ref{fig:dynkin} and marked by an asterisk in Table~\ref{tab:range}; these are the orbits with the largest value of $a$. Explicitly, the endpoint is $\{\alpha_n\}$ in types $\mathrm C_n$, $\mathrm B_n$ and ${}^2\mathrm A_{2n-1}$, the two orbits $\{\alpha_{n-1}\}$ and $\{\alpha_n\}$ in type $\mathrm D_n$, the orbit $\{\alpha_{n-1},\alpha_n\}$ in type ${}^2\mathrm D_n$, and the orbit $\{\alpha_n,\alpha_{n+1}\}$ in type ${}^2\mathrm A_{2n}$. The endpoints behave differently in different types.
\begin{itemize}
\item In types $\mathrm C_n$, $\mathrm D_n$ and ${}^2\mathrm A_{2n-1}$ we have $M=0$, and the radical is the abelian group $\Sym_n(k)$, $\Alt_n(k)$ or $\SHerm_n(K)$, respectively.
\item In types $\mathrm B_n$, ${}^2\mathrm D_n$ and ${}^2\mathrm A_{2n}$ we have $M=1$, $2$ and $1$, respectively, and the radical has the same shape as at the other nodes. Note that the endpoint of type ${}^2\mathrm D_n$ has $a=n-1$: it is the stabilizer of a totally isotropic $(n-1)$-space whose residual plane is anisotropic.
\end{itemize}
If $M>0$, then $\mathbf P_{\mathcal O}$ stabilizes the flag $E\subset E^\perp$, with block sizes $(a,M,a)$, and we call its unipotent radical a \emph{three-block radical}. In the split group $\SO^+_{2n}(q)$ the stabilizer of a totally isotropic $(n-1)$-space corresponds to removing both $\alpha_{n-1}$ and $\alpha_n$, so it is not a maximal parabolic subgroup; its radical nevertheless appears in Section~\ref{sec:orthogonal}.

We write $\mathsf X\in\{\mathrm C,\mathrm B,\mathrm D^+,\mathrm D^-,\mathrm U\}$ for the type of $G$, where $\mathrm D^\pm$ refers to $\SO^\pm_{2n}(q)$ and $\mathrm U$ to unitary groups of either parity, and $R_u^{\mathsf X,a}$ for the unipotent radical of the maximal rational parabolic subgroup with isotropic dimension $a$. When $\mathsf X$ and $a$ are fixed we simply write $U=R_u^{\mathsf X,a}$ and
$$
 N_{U,e}(q)=\bigl|\{\chi\in\Irr(U):\chi(1)=q^e\}\bigr|.
$$

Each such $U$ has nilpotency class at most two. More precisely, in logarithmic coordinates its Lie algebra has the form
$$
 \mathfrak u=V\oplus\mathcal Z,
 \qquad [V,V]\subseteq\mathcal Z,\qquad
 [\mathfrak u,\mathcal Z]=0,
$$
where $V=\Mat_{a\times M}(k)$, or $\Mat_{a\times M}(K)$ in the unitary case, and $\mathcal Z$ is as in Table~\ref{tab:range}. If $M>0$, then $V$ and $\mathcal Z$ are the components of degree one and two of the parabolic root grading (or of its Frobenius-fixed analogue in the twisted cases). If $M=0$, then $V=0$ and $U=\mathcal Z$ is abelian. Using the trace pairings of Section~\ref{sec:characters} we identify $\mathcal Z^*$ with matrices of the same type as those in $\mathcal Z$, and we write $\lambda_T\in\mathcal Z^*$ for the functional attached to such a matrix $T$. Each $T$ gives an alternating form on $V$,
$$
 B_T(X,Y)=\lambda_T([X,Y]),\qquad
 R_T=\rad B_T=\{X\in V:TX=0\},\qquad
 2e_T=\rank_{\BF_q}B_T.
$$
Throughout, $\psi$ is a fixed nontrivial additive character of $\BF_q$.

\subsection{Main results}
Our first result describes all irreducible characters.

\begin{theorem}\label{thm:main}
Let $U=R_u^{\mathsf X,a}$ be one of the radicals above. There is a bijection
$$
 \Irr(U)\longleftrightarrow\coprod_{T\in\mathcal Z^*}R_T^*,
 \qquad R_T^*=\Hom_{\BF_q}(R_T,\BF_q),
$$
where $T$ runs over all elements of $\mathcal Z^*$ (not over congruence classes). The character $\chi_{T,\eta}$ corresponding to $\eta\in R_T^*$ has degree $q^{e_T}$, and
$$
 \chi_{T,\eta}(\exp(X+Z))=
 \begin{cases}
 q^{e_T}\psi\bigl(\lambda_T(Z)+\eta(X)\bigr),&X\in R_T,\\
 0,&X\notin R_T.
 \end{cases}
$$
\end{theorem}

Let $N(a,r)$, $S(a,2s)$ and $C_{\mathrm{sh}}(a,r)$ denote the numbers of $a\times a$ symmetric matrices of rank $r$, of skew-symmetric matrices of rank $2s$, and of skew-Hermitian matrices of rank $r$, respectively; explicit formulas are recalled in Lemma~\ref{lem:counts}.

\begin{theorem}\label{thm:main-degrees}
Suppose that $M>0$; in type $\mathrm C$ the residual dimension is even, and we write $M=2m$. Then the nonzero character-degree multiplicities of $U$ are the following.
\begin{center}
\begin{tabular}{@{}llll@{}}
\toprule
Type & Rank of $T$ & Degree & Number of characters\\
\midrule
$\mathrm C$ & $0\le r\le a$ & $q^{mr}$ & $q^{2m(a-r)}N(a,r)$\\
$\mathrm B,\mathrm D^+,\mathrm D^-$ & $0\le 2s\le a$ & $q^{Ms}$ & $q^{M(a-2s)}S(a,2s)$\\
$\mathrm U$ & $0\le r\le a$ & $q^{Mr}$ & $q^{2M(a-r)}C_{\mathrm{sh}}(a,r)$\\
\bottomrule
\end{tabular}
\end{center}
In particular $N_{U,e}(q)\in\BN[q-1]$ for every $e$, so these families satisfy the analogues of the conjectures of Lehrer and Isaacs. If $M=0$, then $U$ is abelian and all its irreducible characters are linear.
\end{theorem}

The theorem covers in particular the three endpoints with $M>0$, namely those of types $\mathrm B_n$, ${}^2\mathrm D_n$ and ${}^2\mathrm A_{2n}$; for the last one see also Proposition~\ref{unitary:odd-endpoint}.

Recall that a \emph{Brauer pair} is a pair of nonisomorphic groups with identical character tables and identical power maps. Brauer asked whether such pairs exist \cite{Brauer}; the first examples were given by Dade \cite{Dade}, and many more are known, see for instance \cite{EickMueller,Lewis}. The point of the next theorem is not the existence of Brauer pairs as such, but that they arise naturally from parabolic subgroups: the character table of the unipotent radical, even together with its power maps, does not determine the Witt type of the ambient even orthogonal group.

\begin{theorem}\label{thm:intro-witt}
Fix $a$ and $M$, and let $U_{a,J}$ be the orthogonal radicals of Section~\ref{sec:model}, where $J$ runs over the invertible symmetric $M\times M$ matrices.
\begin{enumerate}
\item The character table, the class sizes and the power maps of $U_{a,J}$ do not depend on $J$.
\item If $a\ge3$ and $M$ is even, then $U_{a,J}\simeq U_{a,J'}$ if and only if $J$ and $J'$ have the same Witt type. In all other cases the groups $U_{a,J}$ are all isomorphic.
\end{enumerate}
In particular, let $n\ge4$ and $3\le a\le n-1$. Then the unipotent radicals of the stabilizers of a totally isotropic $a$-space in $\SO^+_{2n}(q)$ and in $\SO^-_{2n}(q)$ form a Brauer pair of order $q^{2a(n-a)+a(a-1)/2}$. If $a\le n-2$, both stabilizers are maximal parabolic subgroups; if $a=n-1$, the second one is the endpoint of type ${}^2\mathrm D_n$ and the first one is not maximal.
\end{theorem}

The smallest member of the family has $n=4$, $a=3$ and order $q^9$; the smallest member in which both stabilizers are maximal has $n=5$, $a=3$ and order $q^{15}$.

Finally, the character parameters can be organized geometrically, which explains the positivity in Theorem~\ref{thm:main-degrees}.

\begin{theorem}\label{thm:intro-geometry}
Let $d=\dim_{\BF_q}V$, and let $\mathcal Y_e\subseteq\mathcal Z^*$ be the reduced locally closed locus where $\rank B_T=2e$.
\begin{enumerate}
\item The spaces $R_T$ form a vector bundle $\mathcal R_e$ of rank $d-2e$ over $\mathcal Y_e$, and the rational points of $\mathcal M_e=\operatorname{Tot}(\mathcal R_e^\vee)$ correspond bijectively to the irreducible characters of $U$ of degree $q^e$. Hence
$$
 N_{U,e}(q)=|\mathcal M_e(\BF_q)|=q^{d-2e}|\mathcal Y_e(\BF_q)|.
$$
\item Suppose $M>0$. Then every nonempty $\mathcal Y_e$ is an exact-rank locus of symmetric, skew-symmetric or skew-Hermitian matrices. The symmetric and skew-symmetric bounded-rank loci are the intersections, as reduced subvarieties, of Schubert varieties with the opposite big cell in a Lagrangian or orthogonal Grassmannian, and the point counts of the exact-rank loci are finite sums of $R$-polynomials of type $\mathrm C$ or $\mathrm D$.
\item In the unitary case the rank loci are the $\BF_q$-forms, for the unitary Frobenius, of ordinary determinantal varieties, and their point counts are finite sums of $R$-polynomials of $W(\mathrm C_a)$ with the unequal parameters $q^2,\ldots,q^2,q$.
\end{enumerate}
Via Deodhar's decomposition in the bilinear cases, and via the expansion of the unequal-parameter $R$-polynomials along weighted distinguished subexpressions in the unitary case, these formulas express the coefficients of $N_{U,e}(q)$ in powers of $q-1$ as sums of products of binomial coefficients (and powers of $2$ in the unitary case); in particular they are nonnegative. If $M=0$, then $\mathcal Y_0=\mathcal Z^*$ and $U$ is abelian.
\end{theorem}

\subsection{Methods and relation to earlier work}
Theorems~\ref{thm:main} and~\ref{thm:main-degrees} rest on well-known tools, and we want to be explicit about this. Since $U$ has nilpotency class at most two and $p$ is odd, the Lazard correspondence \cite{Lazard} identifies $U$ with its Lie algebra, and the Kirillov orbit method \cite{OV} describes $\Irr(U)$ by coadjoint orbits. For groups of class at most two the orbits are easy to determine: a central parameter $T$ defines the alternating form $B_T$, the rank of $B_T$ gives the degree, and the radical $R_T$ carries the remaining linear parameter. The multiplicities then follow from the classical counts of symmetric, skew-symmetric and Hermitian matrices of given rank over a finite field \cite{Carlitz,CarlitzHodges,MacWilliams}, and $(q-1)$-positivity can be read off from the product formulas after substituting $q=1+t$. The description of symmetric and skew-symmetric rank loci by Schubert varieties in opposite cells, together with a count of their points by parabolic $R$-polynomials, already appears in \cite[Section~3.3]{LewisEtAl}.

Beyond these standard ingredients, the paper contains the following.
\begin{enumerate}
\item Explicit character values and degree multiplicities for all three-block radicals of maximal rational parabolic subgroups, including type $\mathrm B$, the non-split orthogonal groups and the unitary groups of both parities (Theorems~\ref{thm:main} and~\ref{thm:main-degrees}). The method is the standard one; what we add is the uniform, explicit result.
\item For the orthogonal radicals, the character tables and power maps do not distinguish the Witt type of the residual form (Theorem~\ref{orth:table}), whereas the groups themselves distinguish the two Witt types when $a\ge3$ and $M>0$ is even: counting rank-one elements with a large set of commuting rank-one neighbours detects the isotropic vectors of the residual form (Theorem~\ref{orth:family}). Together these give the Brauer pairs of Theorem~\ref{thm:intro-witt}.
\item The character parameters of a fixed degree form a vector bundle over a matrix rank locus (Proposition~\ref{prop:parameter-varieties}).
\item In the unitary case we count the rational points of the skew-Hermitian rank loci by folding the Bruhat decomposition of $\GL_{2a}$ along the unitary Frobenius (Theorem~\ref{geo:unitary-count}). The simple reflections that are interchanged in pairs carry weight two and the fixed middle reflection carries weight one, so that $R$-polynomials of type $\mathrm C_a$ with parameters $q^2,\ldots,q^2,q$ appear; expanding them along weighted distinguished subexpressions gives the coefficients in powers of $q-1$.
\end{enumerate}

In the symplectic, split even orthogonal and even-dimensional unitary cases one can also proceed in a more classical way. A hyperbolic decomposition of the residual space gives a factorization $U=A\rtimes H$ with $A$ and $H$ abelian; the trace pairing identifies $\widehat A$ with pairs of matrices $(T,Y)$, on which $H$ acts by
$$
 (T,Y)\longmapsto(T,Y-TP),
$$
and Clifford theory produces the irreducible characters as induced characters. We describe this model in Section~\ref{sec:clifford} and check that it agrees with the orbit method. It is not available in general: in type $\mathrm B$ with $a=n=3$, and at the endpoints $a=n\ge2$ of $\Ugrp_{2n+1}(q)$, the radical is not a product of two abelian subgroups.

\subsection{Organization and notation}
Section~\ref{sec:model} introduces the matrix models and determines which of them come from maximal rational parabolic subgroups. In Section~\ref{sec:characters} we construct all irreducible characters by the orbit method, and in Section~\ref{sec:counts} we determine their degrees and multiplicities. Section~\ref{sec:clifford} gives the induced models in the hyperbolic cases. In Section~\ref{sec:orthogonal} we compare orthogonal radicals of different Witt types and prove Theorem~\ref{thm:intro-witt}. Sections~\ref{geo:split-section} and~\ref{geo:unitary-section} interpret the rank counts in terms of Schubert varieties and $R$-polynomials.

Throughout, $\BN$ is the set of nonnegative integers, and $|S|$ is the cardinality of a finite set $S$. For a vector space $R$ over a field $k$ we write $R^*=\Hom_k(R,k)$. The Lie algebra of an algebraic group $A$ is $\Lie(A)$. For a matrix $A$ we write $A^t$, $A^{-t}$ and $A^*$ for its transpose, inverse transpose and conjugate transpose.

\section{Matrix models}\label{sec:model}

Let $k=\BF_q$ and $K=\BF_{q^2}$, with $q$ odd. We first describe matrix models of the radicals in the bilinear and the unitary cases, and then determine which of them arise from maximal rational parabolic subgroups.

\subsection{The bilinear cases}
Let $\epsilon=-1$ in the symplectic case and $\epsilon=1$ in the orthogonal case. Let $E=k^a$, and let $W=k^M$ carry a nondegenerate bilinear form with Gram matrix $J$, where $J^t=\epsilon J$. On $E\oplus W\oplus E^\vee$ we use the form with Gram matrix
$$
 \mathcal H=\begin{pmatrix}0&0&I_a\\0&J&0\\\epsilon I_a&0&0\end{pmatrix},
 \qquad
 G_{\mathcal H}=\{g\in\GL_{2a+M}(k):g^t\mathcal Hg=\mathcal H\}.
$$
Let $P$ be the stabilizer of $E$ in $G_{\mathcal H}$, and let $U_{a,J}$ be the kernel of the action of $P$ on the graded space associated with the flag
$$
 0\subset E\subset E^\perp\subset E\oplus W\oplus E^\vee.
$$
Thus $U_{a,J}$ is the group of isometries that are block upper unitriangular with respect to $E\oplus W\oplus E^\vee$. Put
$$
 V=\Mat_{a\times M}(k),\qquad
 \mathcal Z=\{Z\in\Mat_a(k):Z^t=-\epsilon Z\},
$$
so that $\mathcal Z=\Sym_a(k)$ in the symplectic case and $\mathcal Z=\Alt_a(k)$ in the orthogonal case. For $X\in V$ and $Z\in\mathcal Z$ let
\begin{equation}\label{eq:L}
 L(X,Z)=\begin{pmatrix}0&X&Z\\0&0&-J^{-1}X^t\\0&0&0\end{pmatrix}.
\end{equation}

\begin{proposition}\label{prop:model}
The maps
$$
 \exp L=1+L+\tfrac12L^2,\qquad
 \log(1+N)=N-\tfrac12N^2
$$
are mutually inverse bijections between $\{L(X,Z):X\in V,\ Z\in\mathcal Z\}$ and $U_{a,J}$. Write $u_J(X,Z)=\exp L(X,Z)$. Then
\begin{equation}\label{eq:BCH}
 u_J(X,Z)\,u_J(Y,Z')=u_J\bigl(X+Y,\,Z+Z'+\tfrac12\beta_J(X,Y)\bigr),
 \qquad
 \beta_J(X,Y)=YJ^{-1}X^t-XJ^{-1}Y^t.
\end{equation}
Moreover $[L(X,Z),L(Y,Z')]=L(0,\beta_J(X,Y))$, and the commutator $uvu^{-1}v^{-1}$ of $u=u_J(X,Z)$ and $v=u_J(Y,Z')$ is $u_J(0,\beta_J(X,Y))$.
\end{proposition}

\begin{proof}
Write a block upper unitriangular matrix as
$$
 g=\begin{pmatrix}1&X&C\\0&1&D\\0&0&1\end{pmatrix}.
$$
A direct computation shows that $g^t\mathcal Hg=\mathcal H$ if and only if
$$
 D=-J^{-1}X^t\quad\text{and}\quad C=Z-\tfrac12XJ^{-1}X^t\ \text{ for some } Z \text{ with } Z^t=-\epsilon Z.
$$
For such $g$ one checks that $\log g=L(X,Z)$, which proves the first assertion. Since the product of any three strictly block upper triangular matrices vanishes, the Baker--Campbell--Hausdorff formula \cite{Lazard} reduces to
\begin{equation}\label{eq:BCH-truncated}
 \exp(A)\exp(B)=\exp\bigl(A+B+\tfrac12[A,B]\bigr),
\end{equation}
and multiplying out the blocks gives $[L(X,Z),L(Y,Z')]=L(0,\beta_J(X,Y))$. This yields~\eqref{eq:BCH} and the formula for the commutator.
\end{proof}

By Proposition~\ref{prop:model}, the elements $u_J(0,Z)$ form a central subgroup containing every commutator, so
$$
 [U_{a,J},U_{a,J}]\subseteq Z(U_{a,J})
$$
and $U_{a,J}$ has nilpotency class at most two. (We say that a group has class two if its commutator subgroup is nontrivial and central.) Let us see when the class is exactly two. Assume $M>0$ and choose row vectors $u,v\in W$ with $c=uJ^{-1}v^t\ne0$. In the symplectic case, let $X$ and $Y$ have rows $u$ and $v$, respectively, in a common position $i$, and zeros elsewhere. Then
$$
 \beta_J(X,Y)=-2cE_{ii}\ne0,
$$
so the symplectic radical has class two for every $a\ge1$. In the orthogonal case $\mathcal Z=0$ when $a=1$, and the group is abelian. If $a\ge2$, let $X$ have row $u$ in position $i$ and $Y$ have row $v$ in position $j\ne i$, with zeros elsewhere. Then
$$
 \beta_J(X,Y)=c(E_{ji}-E_{ij}).
$$
As $i$ and $j$ vary these matrices span $\Alt_a(k)$; in particular the commutator subgroup is nontrivial and the orthogonal radical has class two.

\subsection{The unitary case}
Now let $W=K^M$ carry a nondegenerate Hermitian form with Gram matrix $J=J^*$. On $K^a\oplus W\oplus K^a$ we use the Hermitian form with Gram matrix $\mathcal H$ as above, with $\epsilon=1$, and put
$$
 G_{\mathcal H}=\{g\in\GL_{2a+M}(K):g^*\mathcal Hg=\mathcal H\}.
$$
We define $P$ as the stabilizer of the first summand $K^a$ and $U_{a,J}$ as in the bilinear case, and we set
$$
 V=\Mat_{a\times M}(K),\qquad
 \mathcal Z=\SHerm_a(K)=\{Z\in\Mat_a(K):Z^*=-Z\}.
$$
Over $k$ these spaces have dimensions $2aM$ and $a^2$. Replacing transposes by conjugate transposes in~\eqref{eq:L}, we obtain the corresponding nilpotent Lie algebra, and Proposition~\ref{prop:model} holds verbatim with
$$
 \beta_J(X,Y)=YJ^{-1}X^*-XJ^{-1}Y^*.
$$

If $a>0$ and $M>0$, the brackets span $\mathcal Z$ over $k$. To see this, recall that a nondegenerate Hermitian form has an anisotropic vector, so we may choose $w\in W$ with
$$
 d=wJ^{-1}w^*\in k^\times.
$$
For $i<j$ and $z\in K$, let $X$ have $w$ in row $i$ and $Y$ have $-\bar zd^{-1}w$ in row $j$, with all other rows zero. Then
$$
 \beta_J(X,Y)=zE_{ij}-\bar zE_{ji},
$$
which accounts for the off-diagonal part of $\mathcal Z$. For the diagonal part, note that $c\mapsto c-\bar c$ maps $K$ onto the one-dimensional $k$-space of skew elements of $K$. If $X$ and $Y$ have $w$ and $cw$, respectively, in row $i$ and zeros elsewhere, then
$$
 \beta_J(X,Y)=d(c-\bar c)E_{ii}.
$$
Hence the $k$-linear map $\bigwedge_k^2V\to\mathcal Z$ induced by $\beta_J$ is surjective, that is, $[V,V]=\mathcal Z$. (We do not claim that every element of $\mathcal Z$ is a single bracket.)

\subsection{Maximal rational parabolic subgroups}
We now explain why these matrix groups are exactly the radicals we are interested in. First let $\mathbf G$ be a split group of type $\mathrm B$, $\mathrm C$ or $\mathrm D$, with the standard split maximal torus $\mathbf T$ and Borel subgroup $\mathbf B\supseteq\mathbf T$, and let $\mathfrak g=\Lie(\mathbf G)$. We use the simple roots $\alpha_i=e_i-e_{i+1}$ for $i<n$, together with $\alpha_n=2e_n$ in type $\mathrm C$, $\alpha_n=e_n$ in type $\mathrm B$, and
$$
 \alpha_{n-1}=e_{n-1}-e_n,\qquad \alpha_n=e_{n-1}+e_n
$$
in type $\mathrm D$. For $a$ in the ranges listed in the table below, let $P_a\supseteq\mathbf B$ be the standard maximal parabolic subgroup obtained by removing the node $\alpha_a$, and let $\mathfrak u$ be the Lie algebra of its unipotent radical. Writing a positive root as $\gamma=\sum_ic_i(\gamma)\alpha_i$, we define its degree by $\deg_a(\gamma)=c_a(\gamma)$ and put
$$
 \mathfrak u_i=\bigoplus_{\substack{\gamma\in\Phi^+\\ \deg_a(\gamma)=i}}\mathfrak g_\gamma,
 \qquad i>0,
$$
where $\Phi^+$ is the set of positive roots and $\mathfrak g_\gamma$ is the root space of $\gamma$. This is the parabolic root grading
$$
 \mathfrak u=\bigoplus_{i>0}\mathfrak u_i,
 \qquad [\mathfrak u_i,\mathfrak u_j]\subseteq\mathfrak u_{i+j}.
$$
The roots of degree one and two are as follows.
\begin{center}
\begin{tabular}{@{}lll@{}}
\toprule
Type and range & Degree one & Degree two\\
\midrule
$\mathrm C$, $a<n$ & $e_i\pm e_j$, $i\le a<j$ & $e_i+e_j$, $i\le j\le a$\\
$\mathrm D$, $a\le n-2$ & $e_i\pm e_j$, $i\le a<j$ & $e_i+e_j$, $i<j\le a$\\
$\mathrm B$, $a\le n$ & $e_i\pm e_j$, $i\le a<j$; $e_i$, $i\le a$ & $e_i+e_j$, $i<j\le a$\\
\bottomrule
\end{tabular}
\end{center}
In the first row, $i=j$ stands for the root $2e_i$. Writing each positive root in the basis of simple roots, one checks that the table contains every positive root with nonzero $\alpha_a$-coefficient, and that this coefficient is always one or two. Therefore
$$
 \mathfrak u=\mathfrak u_1\oplus\mathfrak u_2,
 \qquad [\mathfrak u_1,\mathfrak u_1]\subseteq\mathfrak u_2,
 \qquad [\mathfrak u,\mathfrak u_2]=0.
$$

On the matrix side, give $V$ degree one and $\mathcal Z$ degree two. Since $[V,V]\subseteq\mathcal Z$ and $[V,\mathcal Z]=[\mathcal Z,\mathcal Z]=0$, this is a grading of the Lie algebra $V\oplus\mathcal Z$. Comparing the positions of the blocks in~\eqref{eq:L} with the standard root subgroups, one sees that the block $X$ in position $(1,2)$ is the sum of the root spaces of degree one and the block $Z$ in position $(1,3)$ is the sum of those of degree two. For instance, in type $\mathrm C$ the bracket of the root spaces of $e_i-e_j$ and $e_i+e_j$, where $i\le a<j$, is the root space of $2e_i$. Thus $V\simeq\mathfrak u_1$ and $\mathcal Z\simeq\mathfrak u_2$, compatibly with the gradings; for the root filtration of parabolic radicals in general, see \cite[Section~17.1]{MalleTesterman}.

For the unitary groups we choose a Witt decomposition
$$
 K^N=E\oplus W\oplus E^\vee,
 \qquad \dim_KE=a,\quad \dim_KW=M=N-2a,
$$
with $W$ nondegenerate. Let $\mathbf G=\GL_N(\overline{k})$, and let
$$
 F(g)=J_N^{-1}\bigl(g^{(q)t}\bigr)^{-1}J_N,
$$
where $J_N$ is the Gram matrix of the Hermitian form and $g^{(q)}$ is obtained from $g$ by raising all entries to the $q$-th power; then $\mathbf G^F\simeq\Ugrp_N(q)$. The parabolic subgroup $\mathbf P_a$ stabilizing the flag $E\subset E^\perp$ is $F$-stable, with block sizes $(a,M,a)$. On the Lie algebra of its unipotent radical we use the $q$-semilinear map
$$
 F_{\mathrm{Lie}}(L)=-J_N^{-1}L^{(q)t}J_N,
 \qquad\text{so that}\qquad F(\exp L)=\exp(F_{\mathrm{Lie}}(L)).
$$
This map describes the $\BF_q$-structure of the Lie algebra; it should not be confused with the differential of $F$, which is zero. It interchanges the two blocks adjacent to the diagonal and preserves the corner block. Its fixed points are exactly the matrices~\eqref{eq:L}, with conjugate transposes in place of transposes, where
$$
 X\in V=\Mat_{a\times M}(K),\qquad
 Z\in\mathcal Z=\SHerm_a(K);
$$
indeed, the fixed-point equations force the second adjacent block to be $-J^{-1}X^*$ and give $Z^*=-Z$. Thus the fixed points of $F_{\mathrm{Lie}}$ form the $k$-Lie algebra of the previous subsection. Putting the adjacent blocks in degree one and the corner block in degree two gives the unitary analogue of the root grading.

On the absolute root system of type $\mathrm A_{N-1}$, the Frobenius $F$ interchanges $\alpha_i$ and $\alpha_{N-i}$ (Figure~\ref{fig:dynkin}). If $N=2n+1$ and $1\leq a\leq n$, removing the orbit $\{\alpha_a,\alpha_{2n+1-a}\}$ gives block sizes $(a,2n+1-2a,a)$. If $N=2n$ and $a<n$, removing $\{\alpha_a,\alpha_{2n-a}\}$ gives block sizes $(a,2n-2a,a)$. These are maximal proper rational parabolic subgroups. At the endpoint of type ${}^2\mathrm A_{2n}$, where $a=n$, the two removed roots are adjacent and $M=1$, so the radical is still a nonabelian three-block group. At the endpoint of type ${}^2\mathrm A_{2n-1}$, again with $a=n$, the removed node is fixed by $F$, the middle block is empty, and the radical is abelian.

Let $\mathbf P$ be any of the stabilizers above. In the bilinear cases, the action of $\mathbf P$ on the graded pieces of
$$
 0\subset E\subset E^\perp\subset E\oplus W\oplus E^\vee
$$
gives the Levi quotient $\GL(E)\times\operatorname{Isom}(W)$ for the full isometry group; for a symplectic or special orthogonal group the second factor is $\Sp(W)$ or $\SO(W)$, respectively. The kernel of this action is the group of block upper unitriangular isometries, and it is $R_u(\mathbf P)$. In the unitary case, $\mathbf P\subset\GL_N(\overline{k})$ has algebraic Levi quotient $\GL_a\times\GL_M\times\GL_a$; the Frobenius $F$ interchanges the two outer factors and acts on the middle one as a unitary Frobenius, so $\mathbf P^F$ has Levi quotient $\GL_a(K)\times\Ugrp_M(q)$. Its unipotent kernel is again the group of block upper unitriangular matrices satisfying the $F$-fixed isometry equations. In all three-block cases we therefore obtain
$$
 R_u(\mathbf P)^F=U_{a,J}=R_u^{\mathsf X,a}.
$$
The same description applies when the residual orthogonal form is non-split. The admissible values of $a$ are those for which the removed node or $F$-orbit defines a maximal rational parabolic subgroup; see \cite[Section~26.1]{MalleTesterman}.

\subsection{Ranges and endpoints}
We now explain the entries of Table~\ref{tab:range} that are not covered by the discussion above.

For $\SO^-_{2n}(q)$ the residual quadratic form is non-split, and $\rho$ interchanges the two terminal nodes $\alpha_{n-1}$ and $\alpha_n$. Removing an orbit $\{\alpha_a\}$ with $a\le n-2$ gives the stabilizer of a totally isotropic $a$-space, whose residual space has dimension $2(n-a)\ge4$ and is again of minus type. Removing the endpoint orbit $\{\alpha_{n-1},\alpha_n\}$ gives a maximal rational parabolic subgroup with $a=n-1$, whose residual plane is anisotropic over $k$; over $\overline{k}$ this parabolic subgroup is not maximal. For this orbit we define the degree of a root to be the sum of its coefficients at $\alpha_{n-1}$ and $\alpha_n$. The roots $e_i-e_n$ and $e_i+e_n$ with $i<n$ have degree one, and the roots $e_i+e_j$ with $i<j<n$ have degree two. The Frobenius preserves both degree spaces, and their fixed points are the $X$- and $Z$-blocks, respectively. So the grading is available at this endpoint as well.

In $\SO^+_{2n}(q)$, on the other hand, $\rho$ is trivial. Removing $\alpha_{n-1}$ or $\alpha_n$ gives the stabilizers of the maximal totally isotropic subspaces in the two families, with $a=n$ and $M=0$. The stabilizer of a totally isotropic $(n-1)$-space $E$ is not maximal. Indeed, if $g\in\SO^+_{2n}(q)$ stabilizes $E$, then $g$ acts on $E^\perp/E$ with determinant $\det g=1$, because the actions on $E$ and on $(E\oplus W\oplus E^\vee)/E^\perp\simeq E^*$ have inverse determinants. So $g$ fixes both singular lines of the hyperbolic plane $E^\perp/E$, and hence stabilizes both maximal totally isotropic subspaces containing $E$. (In $\mathrm O^+_{2n}(q)$ a reflection can interchange these two lines; we always work in the special orthogonal group.) Proposition~\ref{prop:model} applies to the radical of this stabilizer as well, and it will appear in Section~\ref{sec:orthogonal}.

At the endpoints of types $\mathrm C_n$, $\mathrm D_n$ and ${}^2\mathrm A_{2n-1}$ the radicals $\Sym_n(k)$, $\Alt_n(k)$ and $\SHerm_n(K)$ are abelian, of orders $q^{n(n+1)/2}$, $q^{n(n-1)/2}$ and $q^{n^2}$, and all their irreducible characters are linear. By contrast, the endpoint of type ${}^2\mathrm A_{2n}$ has $M=1$ and gives a nonabelian three-block radical of order $q^{n^2+2n}$.

The orders of the radicals in Table~\ref{tab:range} are
$$
 q^{aM+a(a+1)/2},\qquad q^{aM+a(a-1)/2},\qquad q^{2aM+a^2}
$$
in the symplectic, orthogonal and unitary cases, respectively. When the group is abelian, $\mathcal Z$ need not be its whole center; for instance, if $a=1$ in the orthogonal case, then $\mathcal Z=0$ and $U_{a,J}=V$.

Since a unipotent isometry has determinant one, an orthogonal stabilizer and the corresponding special orthogonal stabilizer have the same unipotent radical; for example, nothing changes if $\SO^-_{2n}(q)$ is replaced by $\mathrm O^-_{2n}(q)$. Similarly, the corresponding stabilizers in $\Ugrp_N(q)$ and $\mathrm{SU}_N(q)$ have the same unipotent radical. For the full orthogonal group, the unipotent radical of a stabilizer means that of the corresponding parabolic subgroup of the identity component, that is, the matrix group described above.

\section{Characters of the radicals}\label{sec:characters}

In this section we describe the irreducible characters of all the radicals in Table~\ref{tab:range} at once, using the orbit method for groups of class at most two. We fix a nontrivial additive character $\psi$ of $k$.

\subsection{The orbit method in class at most two}
We recall the orbit method for groups of class at most two; see \cite[Section~3.2]{OV}. For a $k$-vector space $R$ we identify $R^*$ with the group of additive characters of $R$ via $\eta\mapsto\psi\circ\eta$.

\begin{theorem}[{\cite[Theorem~3.1]{OV}}]\label{thm:characters}
Let $\mathfrak u=V\oplus\mathcal Z$ be a finite-dimensional Lie algebra over $k$ such that
$$
 [V,V]\subseteq\mathcal Z,\qquad [\mathfrak u,\mathcal Z]=0,
$$
and let $U$ be the group with underlying set $\mathfrak u$ and multiplication
$$
 u*v=u+v+\tfrac12[u,v].
$$
For $\lambda\in\mathcal Z^*$ put
$$
 B_\lambda(v,w)=\lambda([v,w]),\qquad R_\lambda=\rad B_\lambda,
 \qquad 2e_\lambda=\rank_kB_\lambda.
$$
Then there is a bijection
$$
 \Irr(U)\longleftrightarrow
 \coprod_{\lambda\in\mathcal Z^*}R_\lambda^*.
$$
The character $\chi_{\lambda,\eta}$ corresponding to $(\lambda,\eta)$ has degree $q^{e_\lambda}$ and is given by
\begin{equation}\label{eq:character-values}
 \chi_{\lambda,\eta}(v+z)=
 \begin{cases}
 q^{e_\lambda}\psi(\lambda(z)+\eta(v)),&v\in R_\lambda,\\
 0,&v\notin R_\lambda,
 \end{cases}
 \qquad v\in V,\ z\in\mathcal Z.
\end{equation}
In particular, there are exactly $q^{\dim V-2e_\lambda}$ irreducible characters with central parameter $\lambda$.
\end{theorem}

\begin{proof}
Since $p$ is odd, $U$ has class less than $p$, so \cite[Theorem~3.1]{OV} applies; we only have to make the coadjoint orbits explicit. For $w\in V$, the coadjoint action of $w$ sends $(\xi,\lambda)\in V^*\oplus\mathcal Z^*$ to $(\xi-B_\lambda(w,-),\lambda)$. As $w$ runs over $V$, the functional $B_\lambda(w,-)$ runs over
$$
 R_\lambda^\perp=\{\zeta\in V^*: \zeta|_{R_\lambda}=0\},
$$
so the orbit of $(\xi,\lambda)$ is $(\xi+R_\lambda^\perp)\times\{\lambda\}$. It has $q^{2e_\lambda}$ elements and is determined by $\lambda$ together with $\eta=\xi|_{R_\lambda}$; every $\eta\in R_\lambda^*$ occurs. By Kirillov's formula, the character attached to this orbit is $q^{-e_\lambda}$ times the sum of $\psi$ over the orbit. Orthogonality of additive characters on $R_\lambda^\perp$ then gives~\eqref{eq:character-values}. Finally, the number of labels with first component $\lambda$ is $|R_\lambda^*|=q^{\dim V-2e_\lambda}$.
\end{proof}

\subsection{The commutator form}
We now apply Theorem~\ref{thm:characters} to the groups of Section~\ref{sec:model}. In the bilinear cases we identify $\mathcal Z^*$ with the space of matrices $T$ satisfying $T^t=-\epsilon T$ via
\begin{equation}\label{eq:lambda-bilinear}
 \lambda_T(Z)=\tfrac12\tr(TZ).
\end{equation}
In the unitary case we use
$$
 \lambda_T(Z)=\tr(TZ)\in k,\qquad T,Z\in\SHerm_a(K).
$$
The first pairing is nondegenerate because $q$ is odd. For the second, choose $\delta\in K^\times$ with $\bar\delta=-\delta$. Multiplication by $\delta$ interchanges Hermitian and skew-Hermitian matrices, and the trace pairing on Hermitian matrices is nondegenerate; hence so is the second pairing.

\begin{proposition}\label{prop:rank}
In the bilinear cases, with respect to row-by-row vectorization,
$$
 [B_T]=\epsilon T\otimes J^{-1},\qquad
 R_T=\{X\in V:TX=0\},\qquad \rank_k B_T=M\rank_kT.
$$
In the unitary case,
$$
 B_T(X,Y)=\Tr_{K/k}\tr(TYJ^{-1}X^*),\quad
 R_T=\{X\in V:TX=0\},\quad \rank_k B_T=2M\rank_KT.
$$
\end{proposition}

\begin{proof}
Consider first the bilinear cases, and put $s=\tr(TYJ^{-1}X^t)$. Using $T^t=-\epsilon T$ and $(J^{-1})^t=\epsilon J^{-1}$, we find
$$
 \tr(TXJ^{-1}Y^t)=-s,
$$
and hence $B_T(X,Y)=\frac12(s+s)=s$. If $x_i$ and $y_i$ denote the rows of $X$ and $Y$, then
$$
 B_T(X,Y)=\sum_{i,j}T_{ij}\,y_jJ^{-1}x_i^t
 =\epsilon\sum_{i,j}T_{ij}\,x_iJ^{-1}y_j^t,
$$
which shows that $[B_T]=\epsilon T\otimes J^{-1}$. It also shows that $X\in R_T$ if and only if $T^tXJ^{-1}=0$, that is, $TX=0$. If $r=\rank_kT$, then $X\in R_T$ exactly when every column of $X$ lies in $\ker T$, so
$$
 \dim_kR_T=M(a-r),\qquad \rank_kB_T=aM-\dim_kR_T=Mr.
$$

In the unitary case put $s=\tr(TYJ^{-1}X^*)$. Then
$$
 \bar s=-\tr(TXJ^{-1}Y^*),
$$
so $B_T(X,Y)=s+\bar s=\Tr_{K/k}(s)$. By cyclicity of the trace,
$$
 B_T(X,Y)=\Tr_{K/k}\tr\bigl((J^{-1}X^*T)Y\bigr).
$$
The $k$-bilinear pairing
$$
 \Mat_{M\times a}(K)\times\Mat_{a\times M}(K)\longrightarrow k,
 \qquad (C,Y)\longmapsto\Tr_{K/k}\tr(CY)
$$
is nondegenerate. Taking $C=J^{-1}X^*T$, we see that $B_T(X,Y)=0$ for all $Y$ if and only if $J^{-1}X^*T=0$. Since $J$ is invertible and $T^*=-T$, this is equivalent to $TX=0$. If $r=\rank_KT$, then $R_T$ has $K$-dimension $M(a-r)$, hence $k$-codimension $2Mr$ in $V$.
\end{proof}

Theorem~\ref{thm:characters} and Proposition~\ref{prop:rank} together describe all irreducible characters of the groups in Table~\ref{tab:range}; this proves Theorem~\ref{thm:main}.

\subsection{Conjugacy classes}\label{sec:conjugacy-classes}
The same coordinates give a simple description of the conjugacy classes, which we shall use to count classes and to compare radicals of different Witt types.

For $v\in V$ let $[V,v]$ be the image of the map $w\mapsto[w,v]$. Conjugation by $w+c$ sends $v+z$ to $v+z+[w,v]$, so the conjugacy class of $v+z$ is
$$
 \{v+z+c:c\in[V,v]\},
 \qquad |(v+z)^U|=q^{\dim_k[V,v]}.
$$
Thus the conjugacy classes are parametrized by pairs consisting of $v\in V$ and a coset in $\mathcal Z/[V,v]$, and the values~\eqref{eq:character-values} are visibly constant on them. In particular, the number $k(U)$ of conjugacy classes of $U$ is
\begin{equation}\label{eq:classcount}
 k(U)=\sum_{v\in V}q^{\dim\mathcal Z-\dim[V,v]}.
\end{equation}

\subsection{\texorpdfstring{Type $\mathrm B$}{Type B}}
In type $\mathrm B$ we have $M=2(n-a)+1$, and the results above apply for every admissible $a$. The induced construction of Section~\ref{sec:clifford}, on the other hand, requires a hyperbolic residual space, and in type $\mathrm B$ it is not available in general.

For example, let $a=n=3$. Then $V=k^3$, $\mathcal Z=\Alt_3(k)\simeq\bigwedge^2k^3$, and, up to a nonzero scalar, the bracket is the exterior product
$$
 k^3\times k^3\longrightarrow\bigwedge^2k^3,\qquad (x,y)\longmapsto x\wedge y.
$$
Since $x\wedge y=0$ only when $x$ and $y$ are linearly dependent, every subspace of $V$ on which the bracket vanishes has dimension at most one. The image in $V$ of an abelian subgroup of $U$ spans such a subspace, so it lies in a line. Consequently $U$ is not a product of two abelian subgroups: the images of two such subgroups span at most a plane in the three-dimensional space $V$.

For $a=2$ the situation is different. Here $\mathcal Z=\Alt_2(k)$ is one-dimensional, and after identifying it with $k$ the bracket becomes a nondegenerate alternating form on $V$. If $V=L\oplus L'$ is a decomposition into complementary Lagrangian subspaces, then $U=\exp(L\oplus\mathcal Z)\rtimes\exp(L')$ is a semidirect product of two abelian groups. So such decompositions do exist in some cases of type $\mathrm B$, but only the class-two argument works uniformly.

When $M=1$, that is, $a=n$, we have $V\simeq k^a$ and $\mathcal Z\simeq\bigwedge^2V$, with the bracket given up to a nonzero scalar by $x\wedge y$. In other words, $\mathfrak u$ is the free nilpotent Lie algebra of class two on $a$ generators, whose character degrees were computed in \cite[Proposition~2.4]{OV}.

\subsection{Parameter spaces}\label{sec:parameter-varieties}
We now organize the parameters of Theorem~\ref{thm:characters} geometrically. This reduces the number of characters of a given degree to a point count on an exact-rank locus, multiplied by the number of points in the fiber of a vector bundle.

Let $d=\dim_kV$, and regard $\mathcal Z^*$ as an affine space over $k$. The bracket $\beta:\bigwedge^2V\to\mathcal Z$ has a dual map
$$
 \beta^*:\mathcal Z^*\longrightarrow\bigwedge^2V^*,
 \qquad \beta^*(\lambda)(v,w)=\lambda([v,w])=B_\lambda(v,w).
$$
For an integer $e$ with $0\leq2e\leq d$, let $\mathcal Y_{\leq e}$ be the reduced closed subvariety of $\mathcal Z^*$ whose geometric points satisfy
$$
 \rank B_\lambda\leq2e.
$$
It is the reduction of the subscheme defined by the pullbacks under $\beta^*$ of the Pfaffians of size $2e+2$. With $\mathcal Y_{\leq-1}=\varnothing$ we put
$$
 \mathcal Y_e=\mathcal Y_{\leq e}\setminus\mathcal Y_{\leq e-1}.
$$
Thus $\mathcal Y_e$ is the reduced locally closed locus of central parameters that give characters of degree $q^e$. Let $\mathcal R_e$ be the kernel of the morphism of trivial bundles
$$
 \mathcal Y_e\times V\longrightarrow\mathcal Y_e\times V^*,
 \qquad (\lambda,v)\longmapsto(\lambda,B_\lambda(v,-)).
$$
Its fiber at $\lambda$ is $R_\lambda=\rad B_\lambda$, which has dimension $d-2e$.

\begin{proposition}\label{prop:parameter-varieties}
The kernel $\mathcal R_e$ is a vector bundle of rank $d-2e$ over $\mathcal Y_e$, with fiber $R_\lambda$ at $\lambda\in\mathcal Y_e$. Let $\mathcal M_e=\operatorname{Tot}(\mathcal R_e^\vee)$ be the total space of the dual bundle. Then $\psi$ induces a bijection
$$
 \mathcal M_e(k)\longleftrightarrow
 \{\chi\in\Irr(U):\chi(1)=q^e\}.
$$
In particular,
\begin{equation}\label{eq:parameter-variety}
 N_{U,e}(q)=|\mathcal M_e(k)|
          =q^{d-2e}|\mathcal Y_e(k)|.
\end{equation}
\end{proposition}

\begin{proof}
On the open subset of $\mathcal Y_e$ where a fixed $2e\times2e$ minor of the matrix of $B_\lambda$ is invertible, a change of local frames brings the bundle map into the form $\operatorname{diag}(I_{2e},A)$. Since every geometric fiber has rank exactly $2e$, the entries of $A$ vanish at every geometric point, and since $\mathcal Y_e$ is reduced, $A=0$. Hence the bundle map has the local form $\operatorname{diag}(I_{2e},0)$ and its kernel is locally free of rank $d-2e$. As the minor varies, these open subsets cover $\mathcal Y_e$.

The fiber of $\mathcal R_e^\vee$ at $\lambda$ is $R_\lambda^*$, so a point of $\mathcal M_e(k)$ is a pair $(\lambda,\eta)$ with $\lambda\in\mathcal Y_e(k)$ and $\eta\in R_\lambda^*$. Since $\rank B_\lambda=2e$, Theorem~\ref{thm:characters} gives the bijection and the degree. Finally,
\begin{equation*}
 |\mathcal M_e(k)|
 =\sum_{\lambda\in\mathcal Y_e(k)}|R_\lambda^*|
 =q^{d-2e}|\mathcal Y_e(k)|.\qedhere
\end{equation*}
\end{proof}

If $M=0$, then $V=0$, $\mathcal Y_0=\mathcal Z^*$ and $\mathcal R_0=0$, so $N_{U,0}(q)=|\mathcal Z^*(k)|=|U|$. For $M>0$ the degree of a character is determined by the rank of its central parameter.

By Proposition~\ref{prop:parameter-varieties}, it remains to count the rational points of $\mathcal Y_e$. The varieties $\mathcal Y_e$ and $\mathcal M_e$ and the bundle $\mathcal R_e$ are defined over $k$, and since $k$ is perfect their formation commutes with finite extensions of $k$. Suppose that $\mathcal Y_e$ has a finite partition into locally closed $k$-subvarieties
$$
 \mathcal Y_e=\bigsqcup_{\gamma\in\Gamma}C_\gamma,
 \qquad
 C_\gamma\simeq_k\mathbb A_k^{b_\gamma}
                    \times(\mathbb G_{m,k})^{c_\gamma}.
$$
Then for every finite extension $k'/k$, with $q'=|k'|$,
$$
 |\mathcal Y_e(k')|
 =\sum_{\gamma\in\Gamma}q'^{\,b_\gamma}(q'-1)^{c_\gamma},
 \qquad
 |\mathcal M_e(k')|
 =\sum_{\gamma\in\Gamma}q'^{\,d-2e+b_\gamma}(q'-1)^{c_\gamma}.
$$
With $t=q'-1$, each summand equals $(1+t)^{d-2e+b_\gamma}t^{c_\gamma}\in\BN[t]$, so $|\mathcal M_e(k')|$ is a polynomial in $q'-1$ with nonnegative integer coefficients. If such partitions exist with the same index set $\Gamma$ and the same integers $b_\gamma$ and $c_\gamma$ for all $q$, it follows that $N_{U,e}(q)\in\BN[q-1]$. Deodhar's decomposition provides such partitions in the bilinear cases (Section~\ref{geo:split-section}). In the unitary case no such stratification over $k$ exists in general (Remark~\ref{geo:extensions}). Instead, Section~\ref{geo:unitary-section} counts the rational points of the unitary rank loci directly, using relative Bruhat cells and unequal-parameter $R$-polynomials; this gives a polynomial in $\BN[q-1]$ that is valid for all odd $q$.

\section{Character degrees and polynomiality}\label{sec:counts}

\subsection{Rank counts over finite fields}
Recall that $N(a,r)$, $S(a,2s)$ and $C_{\mathrm{sh}}(a,r)$ denote the numbers of $a\times a$ symmetric matrices of rank $r$, skew-symmetric matrices of rank $2s$, and skew-Hermitian matrices of rank $r$, respectively. The following formulas are due to MacWilliams \cite[Theorem~2]{MacWilliams}, Carlitz \cite{Carlitz} and Carlitz--Hodges \cite{CarlitzHodges}. For the last one, note that multiplication by a nonzero scalar of trace zero is a rank-preserving bijection between Hermitian and skew-Hermitian matrices. In the formula for $N(a,r)$ we write $r=2s$ or $r=2s+1$.

\begin{lemma}\label{lem:counts}
We have
\begin{align}
 N(a,r)&=\gauss{a}{r}{q}\,q^{s(s+1)}
             \prod_{i=1}^{\lceil r/2\rceil}(q^{2i-1}-1),\label{eq:N}\\
 S(a,2s)&=\gauss{a}{2s}{q}\,q^{s(s-1)}
             \prod_{i=1}^{s}(q^{2i-1}-1),\label{eq:S}\\
 C_{\mathrm{sh}}(a,r)&=\gauss{a}{r}{q^2}\,q^{r(r-1)/2}
             \prod_{i=1}^{r}(q^i+(-1)^i).\label{eq:U}
\end{align}
\end{lemma}

\subsection{The number of characters of each degree}
\begin{theorem}\label{thm:degrees}
For the three-block radicals in Table~\ref{tab:range}, the nonzero character-degree multiplicities are
$$
 \begin{aligned}
 N_{U,mr}(q)&=q^{2m(a-r)}N(a,r)
     &&(\mathrm C,\ M=2m,\ 0\le r\le a),\\
 N_{U,Ms}(q)&=q^{M(a-2s)}S(a,2s)
     &&(\mathrm B,\mathrm D^\pm,\ 0\le 2s\le a),\\
 N_{U,Mr}(q)&=q^{2M(a-r)}C_{\mathrm{sh}}(a,r)
     &&(\mathrm U,\ 0\le r\le a).
 \end{aligned}
$$
All other multiplicities vanish, and every multiplicity lies in $\BN[q-1]$.
\end{theorem}

\begin{proof}
By Theorem~\ref{thm:characters}, a central parameter $T$ contributes $|R_T^*|=|R_T|$ characters, each of degree $q^{\rank_k(B_T)/2}$. The formulas therefore follow from Proposition~\ref{prop:rank} and Lemma~\ref{lem:counts}. (In the language of Proposition~\ref{prop:parameter-varieties}, $|R_T|$ is the number of rational points in a fiber of $\mathcal M_e\to\mathcal Y_e$.)

For positivity, let $t=q-1$. Since $\gauss{a}{r}{Q}\in\BN[Q]$, we have
$$
 \gauss{a}{r}{1+t}\in\BN[t],\qquad
 \gauss{a}{r}{(1+t)^2}\in\BN[t],
$$
and $(1+t)^j-1$ and $(1+t)^j+1$ lie in $\BN[t]$ for every $j\ge1$. The factor $q^i+(-1)^i$ in~\eqref{eq:U} is $q^i-1$ for odd $i$ and $q^i+1$ for even $i$. Hence all multiplicities lie in $\BN[q-1]$.
\end{proof}

Theorem~\ref{thm:main-degrees} follows from Theorem~\ref{thm:degrees}. For the abelian endpoints of Section~\ref{sec:model}, all irreducible characters are linear and their number is the order of the group.

\subsection{Unitary groups of both parities}\label{sec:unitary-both}
Theorem~\ref{thm:degrees} treats both parities of $N$ at once. For $M=N-2a>0$ it gives
\begin{equation}\label{unitary:all-count}
 N_{U,Mr}(q)=q^{2M(a-r)+r(r-1)/2}\gauss{a}{r}{q^2}
 \prod_{i=1}^r(q^i+(-1)^i),\qquad 0\le r\le a.
\end{equation}
In particular, the endpoint $a=n$ of $\Ugrp_{2n+1}(q)$ has $M=1$, while the endpoint of $\Ugrp_{2n}(q)$ has $M=0$ and is abelian. We complement this with a description of the conjugacy classes.

\begin{proposition}\label{unitary:classes}
In the unitary model with $M>0$, if $\rank_K X=s$, then the conjugacy class of $u_J(X,Z)$ has $q^{s(2a-s)}$ elements. Consequently,
\begin{equation}\label{unitary:classnumber}
 k(U_{a,J})=
 \sum_{s=0}^{\min(a,M)}
 \gauss{a}{s}{q^2}\gauss{M}{s}{q^2}
 |\GL_s(q^2)|\,q^{(a-s)^2}.
\end{equation}
\end{proposition}

\begin{proof}
By Section~\ref{sec:conjugacy-classes}, the class of $u_J(X,Z)$ has $|[V,X]|$ elements. Since $J$ is invertible, the substitution $H=-YJ^{-1}$ gives
$$
 [V,X]=\{HX^*-XH^*:H\in\Mat_{a\times M}(K)\}=:I_X.
$$
If $X'=gXh$ with $g\in\GL_a(K)$ and $h\in\GL_M(K)$, then writing $H=g^{-1}H'h^*$ we get
$$
 H'X'^*-X'H'^*=g(HX^*-XH^*)g^*,
$$
so $I_{X'}=gI_Xg^*$. Hence $\dim_kI_X$ depends only on $s=\rank_KX$, and we may assume
$$
 X=\begin{pmatrix}I_s&0\\0&0\end{pmatrix}.
$$
Writing
$$
 H=\begin{pmatrix}A&B\\C&D\end{pmatrix},
 \qquad A\in\Mat_s(K),\quad
 C\in\Mat_{(a-s)\times s}(K),
$$
in the corresponding block form, we obtain
$$
 HX^*-XH^*=\begin{pmatrix}A-A^*&-C^*\\C&0\end{pmatrix}.
$$
Therefore $\dim_kI_X=s^2+2s(a-s)=s(2a-s)$, and the class has $q^{s(2a-s)}$ elements.

For fixed $X$, the classes with horizontal component $X$ correspond to the cosets of $I_X$ in $\SHerm_a(K)$. As $\dim_k\SHerm_a(K)=a^2$, there are $q^{a^2-s(2a-s)}=q^{(a-s)^2}$ of them. Finally, the number of matrices of rank $s$ in $\Mat_{a\times M}(K)$ is
$$
 \gauss{a}{s}{q^2}\gauss{M}{s}{q^2}|\GL_s(q^2)|,
$$
and summing over $s$ gives~\eqref{unitary:classnumber}.
\end{proof}

\begin{proposition}\label{unitary:odd-endpoint}
Let $U$ be the radical at the endpoint $a=n$ of $\Ugrp_{2n+1}(q)$. Then $|U|=q^{n^2+2n}$ and $U$ has nilpotency class two. The number of its irreducible characters of degree $q^r$ is
$$
 q^{2(n-r)+r(r-1)/2}\gauss{n}{r}{q^2}
 \prod_{i=1}^r(q^i+(-1)^i),\qquad 0\le r\le n.
$$
The center of $U$ has order $q^{n^2}$, every noncentral conjugacy class has $q^{2n-1}$ elements, and
$$
 k(U)=q^{n^2}+q^{(n-1)^2}(q^{2n}-1).
$$
\end{proposition}

\begin{proof}
Here $M=1$ and $J=(d)$ for some $d\in k^\times$, so
$$
 V=K^n,\qquad \mathcal Z=\SHerm_n(K),\qquad
 \beta_J(x,y)=d^{-1}(yx^*-xy^*).
$$
Since $\dim_kV=2n$ and $\dim_k\mathcal Z=n^2$, we have $|U|=q^{n^2+2n}$, and the class is two because the bracket is nonzero with values in the central summand $\mathcal Z$.

Let $x\ne0$, and suppose that $y\ne0$ satisfies $yx^*=xy^*$. Comparing the images of these rank-one matrices gives $y=cx$ for some $c\in K$, and then $cxx^*=\bar cxx^*$, so $c=\bar c\in k$. Conversely every $y\in kx$ commutes with $x$. Thus the centralizer of $x$ in $V$ is the line $kx$. In particular no nonzero element of $V$ commutes with all of $V$, so
$$
 Z(U)=\mathcal Z,\qquad |Z(U)|=q^{n^2}.
$$
For $x\ne0$, Proposition~\ref{unitary:classes} with $s=1$ shows that the class has $q^{2n-1}$ elements, and hence
$$
 k(U)=q^{n^2}+\frac{q^{n^2+2n}-q^{n^2}}{q^{2n-1}}
     =q^{n^2}+q^{(n-1)^2}(q^{2n}-1).
$$
The character-degree multiplicities are obtained by putting $M=1$ and $a=n$ in~\eqref{unitary:all-count}.
\end{proof}

\begin{remark}\label{unitary:not-product}
For $n\ge2$, the group of Proposition~\ref{unitary:odd-endpoint} is not a product of two abelian subgroups. Indeed, since the centralizer in $V$ of every nonzero $x$ is $kx$, the image in $V$ of an abelian subgroup lies in a $k$-line. Two such images cannot span $V$, whose $k$-dimension is $2n\ge4$.
\end{remark}

\section{Induced models in the hyperbolic cases}\label{sec:clifford}

Theorem~\ref{thm:characters} already gives all irreducible characters. When the residual space is hyperbolic of dimension $M=2m>0$, however, the characters can also be written down as explicitly induced characters, using the little-group method for semidirect products with an abelian normal subgroup \cite[\S8.2, Proposition~25]{Serre}. This covers the symplectic, the split even orthogonal and the even-dimensional unitary cases.

Take
$$
 J=\begin{pmatrix}0&I_m\\\epsilon I_m&0\end{pmatrix},
 \qquad \mathscr M=\Mat_{a\times m}(k_{\mathsf X}),
$$
where $k_{\mathsf X}=k$ in the bilinear cases and $k_{\mathsf X}=K$, $\epsilon=1$, in the unitary case. Write $\dagger$ for the transpose or the conjugate transpose, accordingly. For $P,Q\in\mathscr M$ put
$$
 h(P)=u_J((P,0),0),\qquad a(Q,Z)=u_J((0,Q),Z),
$$
where $(P,Q)$ denotes the $a\times 2m$ matrix with blocks $P$ and $Q$. By~\eqref{eq:BCH}, $H=\{h(P)\}$ and $A=\{a(Q,Z)\}$ are abelian subgroups, $A$ is normal, $U=A\rtimes H$, and
\begin{align}
 h(P)a(Q,Z)h(P)^{-1}
 &=a(Q,Z+QP^\dagger-\epsilon PQ^\dagger),\label{eq:conjugation-A}\\
 u_J((P,Q),Z)
 &=a\bigl(Q,Z+\tfrac12(QP^\dagger-\epsilon PQ^\dagger)\bigr)h(P).
 \label{eq:coordinate-conversion}
\end{align}
We use the central pairing $\lambda_T$ of Section~\ref{sec:characters}, and on $\mathscr M$ the pairing
$$
 \langle Y,Q\rangle_0=
 \begin{cases}
 -\epsilon\tr(Y^tQ),&k_{\mathsf X}=k,\\
 -\Tr_{K/k}\tr(Y^*Q),&k_{\mathsf X}=K.
 \end{cases}
$$
These nondegenerate trace pairings identify $\widehat A$ with the set of characters
\begin{equation}\label{eq:dual-A}
 \alpha_{T,Y}(a(Q,Z))=\psi\bigl(\lambda_T(Z)+\langle Y,Q\rangle_0\bigr).
\end{equation}
With the action $(h\cdot\alpha)(x)=\alpha(h^{-1}xh)$, substituting~\eqref{eq:conjugation-A} gives
\begin{equation}\label{eq:dual-translation}
 h(P)\cdot(T,Y)=(T,Y-TP).
\end{equation}
Hence the $H$-orbits on $\widehat A$ and their stabilizers are described by
\begin{equation}\label{eq:quotient-kernel}
 \mathcal Q_T=\mathscr M/T\mathscr M,\qquad
 \mathscr K_T=\{P\in\mathscr M:TP=0\}.
\end{equation}

\begin{proposition}\label{thm:clifford-parameters}
Let $T\in\mathcal Z^*$, $\bar Y\in\mathcal Q_T$ and $\tau\in\widehat{\mathscr K_T}$. Choose a representative $Y$ of $\bar Y$, and extend $\alpha_{T,Y}$ to a character $\widetilde\alpha_{T,Y}$ of $A\rtimes h(\mathscr K_T)$ that is trivial on $h(\mathscr K_T)$. The induced characters
$$
 \chi_{T,\bar Y,\tau}
 =\Ind_{A\rtimes h(\mathscr K_T)}^U(\widetilde\alpha_{T,Y}\otimes\tau)
$$
are irreducible, pairwise distinct, and exhaust $\Irr(U)$. If $r=\rank_{k_{\mathsf X}}T$, then $\chi_{T,\bar Y,\tau}$ has degree $|k_{\mathsf X}|^{mr}$, and
\begin{equation}\label{eq:induced-values}
 \chi_{T,\bar Y,\tau}(a(Q,Z)h(P))=
 \begin{cases}
 |k_{\mathsf X}|^{mr}\tau(P)\alpha_{T,Y}(a(Q,Z)),&TP=TQ=0,\\
 0,&\text{otherwise}.
 \end{cases}
\end{equation}
Moreover, $\chi_{T,\bar Y,\tau}$ coincides with the character $\chi_{T,\eta}$ of Theorem~\ref{thm:characters}, where $\eta\in R_T^*$ is defined in the proof.
\end{proposition}

\begin{proof}
By~\eqref{eq:dual-translation}, the $H$-orbits on $\widehat A$ are labelled by the pairs $(T,\bar Y)$, and the stabilizer of each point of the orbit is $h(\mathscr K_T)$. The little-group method therefore gives the classification, and the degree is $|H|/|\mathscr K_T|=|k_{\mathsf X}|^{mr}$ because $\dim_{k_{\mathsf X}}\mathscr K_T=m(a-r)$.

To compare with Theorem~\ref{thm:characters}, note that the relation $T^\dagger=-\epsilon T$ shows that the annihilator of $\mathscr K_T$ under the rectangular trace pairing is $T\mathscr M$. Thus $\bar Y\mapsto\langle Y,-\rangle_0$ identifies $\mathcal Q_T$ with $\mathscr K_T^*$ over $k$. Write $\tau=\psi\circ\eta_H$ and put
$$
 \eta(P,Q)=\eta_H(P)+\langle Y,Q\rangle_0
 \qquad\text{on}\qquad R_T=\mathscr K_T\oplus\mathscr K_T.
$$
If $P\in\mathscr K_T$, then $TP=0$ and $P^\dagger T=0$, so $\lambda_T(QP^\dagger-\epsilon PQ^\dagger)=0$ for every $Q\in\mathscr M$. Consequently the inducing subgroup is $\exp(\mathfrak h\oplus\mathcal Z)$, where $\mathfrak h=\mathscr K_T\oplus\mathscr M\subseteq V$ is a maximal isotropic subspace for $B_T$, i.e.\ a polarization, and the inducing character is the character of this subgroup given by a linear functional extending $(\eta,\lambda_T)$. By \cite[Theorem~3.1(2)]{OV}, the induced character is $\chi_{T,\eta}$. Formula~\eqref{eq:induced-values} now follows from~\eqref{eq:character-values} and~\eqref{eq:coordinate-conversion}.
\end{proof}

For fixed $T$ there are $|k_{\mathsf X}|^{2m(a-r)}$ labels $(\bar Y,\tau)$, in agreement with Theorem~\ref{thm:degrees}.

\section{Witt-type independence and Brauer pairs}\label{sec:orthogonal}

Throughout this section $J$ is an invertible symmetric matrix of size $M\geq1$ over $k=\BF_q$, $q$ odd, and $U_{a,J}$ is the orthogonal radical with logarithmic coordinates
$$
 (X,Z)\in V\oplus\mathcal Z,
 \qquad V=\Mat_{a\times M}(k),\qquad \mathcal Z=\Alt_a(k),
$$
and bracket
$$
 \beta_J(X,Y)=YJ^{-1}X^t-XJ^{-1}Y^t.
$$
As before, $u_J(X,Z)=\exp L(X,Z)$; its upper right block is $Z-\frac12XJ^{-1}X^t$. On pure tensors $X=ev^t$ and $Y=fw^t$, with $e,f\in k^a$ and $v,w\in k^M$, the bracket is
$$
 \beta_J(ev^t,fw^t)=(v^tJ^{-1}w)\,(fe^t-ef^t),
$$
that is, the symmetric form $J^{-1}$ on $k^M$ tensored with the exterior product on $k^a$. For $p$-groups of class two studied through their commutator maps, including tensor-product constructions and groups of orthogonal type, see Wilson \cite[Section~7]{Wilson}.

Theorem~\ref{thm:degrees} already covers the non-split orthogonal case: for $M=2(n-a)$ the number of characters of degree $q^{Ms}$ is $q^{M(a-2s)}S(a,2s)$. At $a=n-1$ the two terminal nodes form a single Frobenius orbit, which gives a maximal rational parabolic subgroup with anisotropic residual plane, whereas the split stabilizer of the same isotropic dimension is not maximal (Section~\ref{sec:model}). In this section we compare the groups themselves: they share their character tables and power maps, but they are not always isomorphic.

\subsection{Independence of the character table}

\begin{theorem}\label{orth:table}
Fix $a\geq1$ and $M\geq1$. If $J$ and $J'$ are invertible symmetric $M\times M$ matrices, then there are explicit bijections between the irreducible characters and between the conjugacy classes of $U_{a,J}$ and $U_{a,J'}$ under which the two character tables coincide. The bijection between classes preserves class sizes and all power maps.
\end{theorem}

\begin{proof}
We identify the underlying sets of the two groups through their logarithmic coordinates. For $X\in V$ let
$$
 I_X=\{HX^t-XH^t:H\in\Mat_{a\times M}(k)\}\subseteq\mathcal Z.
$$
Taking $H=YJ^{-1}$ we have $\beta_J(X,Y)=HX^t-XH^t$, and $H$ runs over all of $\Mat_{a\times M}(k)$ as $Y$ does, because $J$ is invertible. By Section~\ref{sec:conjugacy-classes}, the conjugacy class of $u_J(X,Z)$ is therefore
\begin{equation}\label{orth:classes}
 \mathcal C_J(X,Z+I_X):=\{u_J(X,Z+D):D\in I_X\},
\end{equation}
and $I_X$ does not depend on $J$. Hence
$$
 \mathcal C_J(X,Z+I_X)\longmapsto\mathcal C_{J'}(X,Z+I_X)
$$
is a well-defined bijection between the conjugacy classes of the two groups, and corresponding classes both have $|I_X|$ elements.

By Theorem~\ref{thm:characters} and Proposition~\ref{prop:rank}, the irreducible characters of both groups are parametrized by the same pairs
$$
 (T,\eta),\qquad T\in\Alt_a(k),\quad
 \eta\in\{X:TX=0\}^*,
$$
and, writing $\rank T=2s$,
$$
 \chi_{T,\eta}(u_J(X,Z))=
 \begin{cases}
 q^{Ms}\psi\bigl(\frac12\tr(TZ)+\eta(X)\bigr),&TX=0,\\
 0,&TX\ne0.
 \end{cases}
$$
Neither the parameters nor the values depend on $J$, so the character tables coincide. Finally, since $u_J(X,Z)=\exp L(X,Z)$,
$$
 u_J(X,Z)^j=\exp(jL(X,Z))=u_J(jX,jZ)
$$
for every integer $j$, where on the right $j$ is read in $k$. As $I_{jX}=jI_X$, the bijection between classes commutes with all power maps.
\end{proof}

\begin{proposition}\label{orth:classcount}
If $\rank X=r$, then the conjugacy class of $u_J(X,Z)$ has $q^{ar-r(r+1)/2}$ elements. Consequently,
\begin{equation}\label{orth:classnumber}
 k(U_{a,J})=
 \sum_{r=0}^{\min(a,M)}
 \gauss{a}{r}{q}\gauss{M}{r}{q}
 |\GL_r(q)|\,q^{\binom{a-r}{2}}.
\end{equation}
\end{proposition}

\begin{proof}
The class has $|I_X|$ elements. If $X'=gXh$ with $g\in\GL_a(k)$ and $h\in\GL_M(k)$, put $H'=gHh^{-t}$; then
$$
 H'X'^t-X'H'^t=g(HX^t-XH^t)g^t,
$$
so $I_{X'}=gI_Xg^t$. We may therefore assume $X=\left(\begin{smallmatrix}I_r&0\\0&0\end{smallmatrix}\right)$. Writing $H=\left(\begin{smallmatrix}A&B\\C&D\end{smallmatrix}\right)$ in the same block form, we get
$$
 HX^t-XH^t=\begin{pmatrix}A-A^t&-C^t\\C&0\end{pmatrix},
 \qquad C\in\Mat_{(a-r)\times r}(k),
$$
and hence
$$
 \dim I_X=\binom r2+r(a-r)=\binom a2-\binom{a-r}{2}
           =ar-\frac{r(r+1)}2.
$$
For fixed $X$ the classes are parametrized by $\Alt_a(k)/I_X$, a space of dimension $\binom{a-r}{2}$, so there are $q^{\binom{a-r}{2}}$ classes with horizontal component $X$. Since the number of $a\times M$ matrices of rank $r$ is
$$
 \gauss{a}{r}{q}\gauss{M}{r}{q}|\GL_r(q)|,
$$
summing over $r$ gives~\eqref{orth:classnumber}.
\end{proof}

\subsection{Isomorphism classes}

\begin{lemma}\label{orth:similar}
Let $J$ and $J'$ be invertible symmetric $M\times M$ matrices with $J'=c\,h^tJh$ for some $c\in k^\times$ and $h\in\GL_M(k)$. Then
$$
 u_J(X,Z)\longmapsto u_{J'}(Xh,\,c^{-1}Z)
$$
is an isomorphism $U_{a,J}\to U_{a,J'}$.
\end{lemma}

\begin{proof}
Since $J'^{-1}=c^{-1}h^{-1}J^{-1}h^{-t}$, we have $\beta_{J'}(Xh,Yh)=c^{-1}\beta_J(X,Y)$, and the claim follows from~\eqref{eq:BCH}.
\end{proof}

We call $J$ and $J'$ \emph{similar} if they are related as in Lemma~\ref{orth:similar}. Since $q$ is odd, nondegenerate symmetric bilinear forms over $k$ are classified up to isometry by their dimension and their discriminant \cite[Chapter~IV, \S1]{SerreArith}. If $M$ is odd, multiplying $J$ by a nonsquare changes the discriminant, so any two invertible symmetric $M\times M$ matrices are similar. If $M=2m$ is even, multiplying by a scalar does not change the discriminant class, and there are exactly two similarity classes: the split forms $J_+$, of Witt index $m$, and the non-split forms $J_-$, of Witt index $m-1$. A standard count, splitting off hyperbolic planes, shows that the number of nonzero isotropic vectors of $J_\pm$ is
\begin{equation}\label{orth:isotropic}
 n_\pm(m)=(q^m\mp1)(q^{m-1}\pm1),
 \qquad n_+(m)-n_-(m)=2q^{m-1}(q-1)\ne0.
\end{equation}

\begin{theorem}\label{orth:family}
Let $a\ge3$, and let $J$ and $J'$ be invertible symmetric $M\times M$ matrices. Then $U_{a,J}\simeq U_{a,J'}$ if and only if $J$ and $J'$ are similar. In particular, if $M=2m$ is even, then $U_{a,J_+}$ and $U_{a,J_-}$ form a Brauer pair of order $q^{2am+a(a-1)/2}$.
\end{theorem}

\begin{proof}
By Lemma~\ref{orth:similar} it suffices to show that $U_{a,J_+}\not\simeq U_{a,J_-}$ when $M=2m$ is even; the Brauer pair statement then follows from Theorem~\ref{orth:table}. We construct an invariant of the abstract group $U=U_{a,J}$ that depends on the Witt type of $J$.

Let $\pi:U\to V$ be the projection onto the horizontal coordinate, put $Q=J^{-1}$, and for $x\in U$ write $\operatorname{rk}(x)=\rank\pi(x)$. By Proposition~\ref{orth:classcount}, $|x^U|=q^{f(\operatorname{rk}(x))}$ with $f(r)=ar-r(r+1)/2$. For $2\le r\le a$ we have
$$
 f(r)-f(1)=\tfrac12(r-1)(2a-r-2)>0,
$$
because $a\ge3$. Hence $\operatorname{rk}(x)\le1$ if and only if $|x^U|\le q^{a-1}$, so the condition $\operatorname{rk}(x)\le1$ can be expressed in terms of the group structure alone. For $x\in U$ with $\operatorname{rk}(x)=1$ put
$$
 S(x)=\{y\in U:\ xy=yx,\ \operatorname{rk}(y)\le1,\ \operatorname{rk}(xy)\le1\}.
$$
Equivalently,
$$
 S(x)=\{y\in U:\ xy=yx,\ |y^U|\le q^{a-1},\ |(xy)^U|\le q^{a-1}\},
$$
and $\operatorname{rk}(x)=1$ if and only if $|x^U|=q^{a-1}$. Hence $|S(x)|$ is preserved by every isomorphism of abstract groups; no compatibility with the coordinates, the tensor decomposition or the $k$-linear structure is needed. This matters when $q$ is not prime.
By Proposition~\ref{prop:model}, $x$ and $y$ commute if and only if $\beta_J(\pi(x),\pi(y))=0$, and $\pi(xy)=\pi(x)+\pi(y)$. So the conditions defining $S(x)$ depend only on $\pi(y)$, and $|S(x)|=q^{\binom a2}|\pi(S(x))|$.

Write $\pi(x)=ev^t$ with nonzero column vectors $e\in k^a$ and $v\in k^M$. We claim that
$$
 \pi(S(x))=
 \begin{cases}
 e\otimes k^M,&v^tQv\ne0,\\
 (e\otimes k^M)\cup(k^a\otimes v),&v^tQv=0,
 \end{cases}
$$
where $e\otimes k^M=\{ew^t:w\in k^M\}$ and $k^a\otimes v=\{fv^t:f\in k^a\}$. Let $Y=fw^t\ne0$ be a matrix of rank one. Then
$$
 \beta_J(ev^t,fw^t)=(v^tQw)(fe^t-ef^t),
$$
so $Y$ commutes with $\pi(x)$ if and only if $f\in ke$ or $v^tQw=0$. Moreover $ev^t+fw^t$ has rank at most one if and only if $f\in ke$ or $w\in kv$. If $f\in ke$, both conditions hold. If $f\notin ke$, they hold exactly when $w=cv$ with $c\,v^tQv=0$; since $Y\ne0$, this is possible only if $v^tQv=0$, and then $Y\in k^a\otimes v$. This proves the claim. The two subspaces $e\otimes k^M$ and $k^a\otimes v$ meet in the line $k\,ev^t$, and therefore
$$
 |S(x)|=
 \begin{cases}
 q^{\binom a2}q^M,&v^tQv\ne0,\\
 q^{\binom a2}(q^M+q^a-q),&v^tQv=0.
 \end{cases}
$$
These two values are different, since $a\ge2$.

Let $\nu(U)$ be the number of elements $x\in U$ with $|x^U|=q^{a-1}$ and $|S(x)|=q^{\binom a2}(q^M+q^a-q)$. This number is preserved by group isomorphisms. By the computation above, $\nu(U)$ is $q^{\binom a2}$ times the number of rank-one matrices $ev^t$ with $v^tQv=0$. A rank-one matrix has exactly $q-1$ factorizations $ev^t$, and $Q=J^{-1}JJ^{-1}$ is congruent to $J$, so
$$
 \nu(U_{a,J})=q^{\binom a2}\,\frac{q^a-1}{q-1}\,n_\pm(m)
 \qquad\text{if $J$ is of type $\pm$.}
$$
By~\eqref{orth:isotropic} these numbers are different, so $U_{a,J_+}\not\simeq U_{a,J_-}$.
\end{proof}

\begin{corollary}\label{orth:family-parabolic}
Let $n\ge4$ and $3\le a\le n-1$, and let $U^\pm$ be the unipotent radical of the stabilizer of a totally isotropic $a$-dimensional subspace in $\SO^\pm_{2n}(q)$. Then $U^+$ and $U^-$ form a Brauer pair of order $q^{2a(n-a)+a(a-1)/2}$. If $a\le n-2$, both stabilizers are maximal parabolic subgroups; if $a=n-1$, then $U^-$ is the radical at the endpoint of type ${}^2\mathrm D_n$, and the stabilizer in $\SO^+_{2n}(q)$ is not maximal.
\end{corollary}

\begin{proof}
By Section~\ref{sec:model}, $U^\pm\simeq U_{a,J_\pm}$ with $M=2(n-a)$, the residual space being split for $\SO^+_{2n}(q)$ and non-split for $\SO^-_{2n}(q)$; the statements about maximality are in Section~\ref{sec:model} as well. Now apply Theorem~\ref{orth:family}.
\end{proof}

\begin{remark}\label{orth:comparison}
The hypothesis $a\ge3$ in Theorem~\ref{orth:family} cannot be dropped. For $a=1$ the groups $U_{1,J}$ are abelian of order $q^M$. For $a=2$ the commutator is a nondegenerate alternating form on $V\simeq k^{2M}$ with values in the line $\mathcal Z$; any two such forms are equivalent, and an isometry between them, extended by the identity on $\mathcal Z$, is an isomorphism. Together with Lemma~\ref{orth:similar} and Theorem~\ref{orth:family}, this proves Theorem~\ref{thm:intro-witt}: for $a\ge3$ the groups $U_{a,J}$ form one isomorphism class if $M$ is odd and two if $M$ is even, while for $a\le2$ they are all isomorphic.
\end{remark}

\subsection{Abelian subgroups in the smallest case}
In the smallest case $a=3$, $M=2$ the two groups can also be told apart by a more familiar invariant.

\begin{lemma}\label{orth:rankone}
Let $E$ and $W$ be nonzero finite-dimensional vector spaces over a field, and let $L\subseteq E\otimes W$ be a linear subspace all of whose elements have tensor rank at most one. Then either $L\subseteq e\otimes W$ for some nonzero $e\in E$, or $L\subseteq E\otimes w$ for some nonzero $w\in W$.
\end{lemma}

\begin{proof}
We may assume $L\ne0$; choose $0\ne e_0\otimes w_0\in L$. If $e\otimes w\in L$ is nonzero, then $e\in ke_0$ or $w\in kw_0$, since otherwise $e\otimes w+e_0\otimes w_0$ would have rank two. If all first factors lie in $ke_0$, then $L\subseteq e_0\otimes W$. Otherwise $L$ contains an element $e_1\otimes w_0$ with $e_0,e_1$ linearly independent. Applying the same argument to $e_0\otimes w_0$ and $e_1\otimes w_0$, we see that every second factor lies in $kw_0$, so $L\subseteq E\otimes w_0$.
\end{proof}

\begin{proposition}\label{orth:nonisomorphic}
Let $J_+$ and $J_-$ be Gram matrices of a split and of an anisotropic plane over $k$, respectively. Then the largest abelian subgroups of $U_{3,J_+}$ and $U_{3,J_-}$ have orders $q^6$ and $q^5$, respectively. Both groups have order $q^9$, with $q^6$ linear characters and $q^2(q^3-1)$ characters of degree $q^2$.
\end{proposition}

\begin{proof}
The character degrees follow from Theorem~\ref{thm:degrees} with $a=3$ and $M=2$, and $|U_{3,J}|=|V||\mathcal Z|=q^{6+3}=q^9$.

Put $Q=J^{-1}$. Since $Q=(J^{-1})^tJ(J^{-1})$, the form $Q$ has the same Witt type as $J$. Call a subspace $L\subseteq\Mat_{3\times2}(k)$ \emph{commuting} if $\beta_J(L,L)=0$. We first bound the dimension of a commuting subspace $L$.

Suppose that $L$ contains a matrix $X$ of rank two. For $g\in\GL_3(k)$ and $h\in\GL_2(k)$ we have
$$
 \beta_{h^tJh}(gXh,gYh)=g\beta_J(X,Y)g^t.
$$
Replacing $(J,L,X)$ by $(h^tJh,gLh,gXh)$ preserves the commuting condition, the dimension of $L$ and the Witt type of $J$. We may therefore assume that
$$
 X=\begin{bmatrix}I_2\\0\end{bmatrix}\in L,
$$
where $Q=J^{-1}$ now refers to the new $J$. Writing $Y=\left[\begin{smallmatrix}B\\y\end{smallmatrix}\right]$, the equation $\beta_J(X,Y)=0$ is equivalent to $y=0$ and $BQ=QB^t$. So $L$ is contained in
$$
 S=\left\{\begin{bmatrix}B\\0\end{bmatrix}:B\in\Mat_2(k)\right\},
$$
and the centralizer of $X$ in $V$ is the three-dimensional subspace of $S$ defined by the symmetry of $BQ$. Let $T=E_{12}-E_{21}\in\Alt_3(k)$. On $S$ the bracket takes values in $kT$, and $\lambda_T(T)=-1$. By Proposition~\ref{prop:rank}, the matrix of $B_T|_S$ is
$$
   \begin{bmatrix}0&1\\-1&0\end{bmatrix}\otimes Q,
$$
which is nondegenerate because $Q$ is invertible. Since $L$ is an isotropic subspace for this nondegenerate alternating form on the four-dimensional space $S$, we get $\dim L\leq2$.

Suppose now that every element of $L$ has rank at most one. By Lemma~\ref{orth:rankone}, either $L\subseteq e\otimes k^2$ for some $0\ne e\in k^3$, or $L\subseteq k^3\otimes w$ for some $0\ne w\in k^2$. In the first case $\dim L\leq2$, and $e\otimes k^2$ is indeed commuting, since
$$
 \beta_J(ev^t,ew^t)=e(w^tQv-v^tQw)e^t=0.
$$
In the second case
$$
 \beta_J(ew^t,fw^t)
   =(w^tQw)(fe^t-ef^t).
$$
If $J$, and hence $Q$, is anisotropic, then $w^tQw\ne0$, and $fe^t-ef^t=0$ only when $e$ and $f$ are dependent; so $\dim L\leq1$. If $J$ is split and $w$ is isotropic, then $k^3\otimes w$ is a commuting subspace of dimension three.

Combining the two cases, the largest commuting subspaces have dimension two for $J_-$ and three for $J_+$.

Now let $\pi:U_{3,J}\to V$ be the projection onto the horizontal coordinate. By~\eqref{eq:BCH} it is a homomorphism with kernel $\mathcal Z$, of order $q^3$. If $H\leq U_{3,J}$ is abelian, the bracket vanishes on $\pi(H)$, hence, by bilinearity, on the subspace $L$ spanned by $\pi(H)$. Therefore
$$
 |H|\le|\pi(H)|\,|\mathcal Z|\leq q^{\dim_kL+3},
$$
which gives $|H|\leq q^6$ for $J_+$ and $|H|\leq q^5$ for $J_-$. The first bound is attained by $\exp((k^3\otimes w)\oplus\mathcal Z)$ with $w$ isotropic, and the second by $\exp((e\otimes k^2)\oplus\mathcal Z)$ for any $0\ne e\in k^3$. This gives a second proof that $U_{3,J_+}\not\simeq U_{3,J_-}$.
\end{proof}

\section{Geometric interpretation in the bilinear cases}\label{geo:split-section}

Suppose $M>0$ in a bilinear case. By Proposition~\ref{prop:rank}, every nonempty base $\mathcal Y_e$ is then an exact-rank locus of symmetric or skew-symmetric matrices. The description of the bounded-rank loci by Schubert varieties, and the count of their points by parabolic $R$-polynomials, can be found in \cite[Section~3.3]{LewisEtAl}; here we fix the indexing needed for the exact-rank counts. For a non-split orthogonal radical the parameter space is still $\Alt_a(k)$, and we use the split orthogonal Grassmannian as an auxiliary model for its rank loci. When $a=1$ the skew-symmetric locus is a point, so in the type $\mathrm D$ arguments below we assume $a\ge2$.

\subsection{Rank loci in opposite Schubert cells}\label{geo:indices}

Let $E=\langle e_1,\ldots,e_a\rangle$ and $E^\vee=\langle f_1,\ldots,f_a\rangle$, and equip $E\oplus E^\vee$ with the hyperbolic symplectic or symmetric bilinear form in which $e_i$ and $f_j$ pair to $\delta_{ij}$. Let $G$ be the corresponding split symplectic or special orthogonal group; it is an auxiliary group used only to realize the rank loci, and not the ambient group of the radical. Let $B$ be the Borel subgroup stabilizing the flag defined by the ordered basis
$$
 e_1,\ldots,e_a,f_a,\ldots,f_1,
$$
let $B^-$ be the opposite Borel subgroup, and let $P$ be the stabilizer of $E$. We write $\operatorname{LG}(a,2a)$ for the Lagrangian Grassmannian of $E\oplus E^\vee$, and $\operatorname{OG}(a,2a)$ for the variety of $a$-dimensional totally isotropic subspaces of the split quadratic space $E\oplus E^\vee$, which has two connected components. Then $G/P=\operatorname{LG}(a,2a)$ in type $\mathrm C$, while in type $\mathrm D$ the quotient $G/P$ is the component of $\operatorname{OG}(a,2a)$ containing $E$.

A linear map $T:E\to E^\vee$, identified with its matrix with respect to the bases $(e_i)$ and $(f_i)$, has graph
$$
 \Gamma_T=\{x+Tx:x\in E\},\qquad
 \Gamma_T\cap E=\ker T,
$$
so that $\dim(\Gamma_T\cap E)=a-\rank T$. With respect to the basis $e_1,\ldots,e_a,f_1,\ldots,f_a$ (note that this is not the order defining $B$), the form has Gram matrix $J_\epsilon=\left[\begin{smallmatrix}0&I_a\\\epsilon I_a&0\end{smallmatrix}\right]$, where $\epsilon=-1$ in type $\mathrm C$ and $\epsilon=1$ in type $\mathrm D$. The matrix $u_T=\left[\begin{smallmatrix}I_a&0\\T&I_a\end{smallmatrix}\right]$ maps $E$ onto $\Gamma_T$, and
$$
 u_T^tJ_\epsilon u_T
 =\begin{pmatrix}T+\epsilon T^t&I_a\\\epsilon I_a&0\end{pmatrix}.
$$
The first $a$ columns of $u_T$ form a basis of $\Gamma_T$, so the restriction of the form to $\Gamma_T$ has Gram matrix $T+\epsilon T^t$. Hence $\Gamma_T$ is totally isotropic, equivalently $u_T\in G$, exactly when $T$ is symmetric (type $\mathrm C$) or alternating (type $\mathrm D$). These matrices $u_T$ form the unipotent radical $U_P^-$ of the opposite parabolic subgroup. Since $B^-=U_P^-(B^-\cap P)$, the opposite big cell $\mathcal C^-=B^-P/P=U_P^-P/P$ is identified with $\Sym_a$ in type $\mathrm C$ and with $\Alt_a$ in type $\mathrm D$.

Let $W$ be the Weyl group of $G$, with the simple reflections determined by $B$; it has type $\mathrm C_a$ or $\mathrm D_a$. The Weyl group of the Levi subgroup of $P$ is $W_P\simeq\mathfrak S_a$, and we write $W^P$ for the set of minimal length representatives of $W/W_P$. We write elements of $W$ as signed permutations, in one-line notation with respect to the order $1<\cdots<a<-a<\cdots<-1$, and set
\begin{equation}\label{geo:wr}
 w_r=(r+1,\ldots,a,-r,-r+1,\ldots,-1).
\end{equation}
In type $\mathrm C$ we allow $0\le r\le a$; in type $\mathrm D$, $r$ is even and $0\le r\le2\lfloor a/2\rfloor$. The entries of $w_r$ increase in the signed order, so $w_r\in W^P$. Exactly $a-r$ of its first $a$ entries are positive, so $\dim(w_rE\cap E)=a-r$, which matches $\dim(\Gamma_T\cap E)$ when $\rank T=r$. Since the label $-j$ corresponds to $f_j$, which sits in position $2a-j+1$ of the ordered basis, the corresponding index in the ordinary Grassmannian is
$$
 I_r^{\mathrm{Gr}}=(r+1,\ldots,a,2a-r+1,\ldots,2a).
$$
For $w\in W^P$ let $X_P(w)=\overline{BwP/P}$ be the Schubert variety. Let $\mathcal D_r^{\mathrm C}\subseteq\Sym_a$ and $\mathcal D_r^{\mathrm D}\subseteq\Alt_a$ be the reduced loci of matrices of rank at most $r$, and let
$$
 \mathcal S_r^{\mathrm C}=\{T\in\Sym_a:\rank T=r\},\qquad
 \mathcal S_r^{\mathrm D}=\{T\in\Alt_a:\rank T=r\}
$$
be the exact-rank loci. When the type is fixed we drop the superscripts.

\begin{proposition}\label{geo:rank-chart}
For every allowed $r$ there is an equality of reduced subvarieties
$$
 \mathcal D_r^{\mathrm C}=X_P(w_r)\cap\mathcal C^-,\qquad
 \mathcal D_r^{\mathrm D}=X_P(w_r)\cap\mathcal C^-.
$$
The exact-rank loci $\mathcal S_r^{\mathrm C}$ and $\mathcal S_r^{\mathrm D}$ are smooth of dimension
$$
 \frac{r(2a-r+1)}2\qquad\text{and}\qquad \frac{r(2a-r-1)}2,
$$
respectively. Their closures are the corresponding loci $\mathcal D_r$, and their boundaries are the unions of the loci of smaller allowed rank.
\end{proposition}

\begin{proof}
The correspondence between rank conditions and Schubert conditions is given in \cite[Section~3.3]{LewisEtAl}, with the Grassmannian index $I_r^{\mathrm{Gr}}$; see also \cite[Chapters~6 and~7]{LR}. Under $T\mapsto\Gamma_T$, the incidence condition defining $X_P(w_r)$ becomes $\dim(\Gamma_T\cap E)\ge a-r$, that is, $\rank T\le r$. In type $\mathrm D$ we work in the component containing $E$ and with even $r$. Both sides are reduced (the right-hand side is an open subscheme of a Schubert variety), so the equality of underlying sets is an equality of subvarieties.

The exact-rank stratum is also easy to describe directly. Sending $T$ to $\ker T$ identifies it with the open subset of nondegenerate forms in the vector bundle $\operatorname{Sym}^2\mathcal Q^\vee$, respectively $\bigwedge^2\mathcal Q^\vee$, over $\operatorname{Gr}(a-r,E)$, where $\mathcal Q$ is the tautological quotient bundle. The base has dimension $r(a-r)$ and the fibers have dimension $r(r+1)/2$, respectively $r(r-1)/2$. This gives smoothness, the dimension and density, and shows that the boundary consists of the matrices of smaller allowed rank.
\end{proof}

\begin{remark}\label{geo:orthogonal-components}
In type $\mathrm D$ only the component containing $E$ is used. For $\Lambda$ in this component we have $\dim(\Lambda\cap E)\equiv a\pmod 2$. Since $\rank T=a-\dim(\Gamma_T\cap E)$, only even ranks occur, and no factor of $2$ arises from the two components. For $a=1$ the skew-symmetric chart is a point, and for $a=2$ we use $\mathrm D_2=\mathrm A_1\times\mathrm A_1$. The same chart serves for the non-split orthogonal radicals.
\end{remark}

\subsection{Equations and singular loci}
The locus $\mathcal D_r^{\mathrm C}$ is the reduced zero locus of the $(r+1)\times(r+1)$ minors, and $\mathcal D_{2s}^{\mathrm D}$ is the reduced zero locus of the Pfaffians of size $2s+2$; see \cite[Sections~6.3 and~6.4]{Weyman} for these determinantal ideals.

\begin{proposition}\label{geo:singular}
Let $0<r<a$. Then the singular locus of $\mathcal D_r^{\mathrm C}$ is, as a set, $\mathcal D_{r-1}^{\mathrm C}$. Similarly, if $0<2s<2\lfloor a/2\rfloor$, the singular locus of $\mathcal D_{2s}^{\mathrm D}$ is, as a set, $\mathcal D_{2s-2}^{\mathrm D}$. The loci of rank zero and of the largest allowed rank bound are smooth.
\end{proposition}

\begin{proof}
We give a direct argument, valid in every odd characteristic. A symmetric or alternating matrix of rank $\rho$ has an invertible principal $\rho\times\rho$ submatrix, and taking the Schur complement with respect to it identifies the rank-at-most-$r$ locus, near such a point, with a smooth factor times the rank-at-most-$(r-\rho)$ locus of the same symmetry type in size $a-\rho$. For $\rho=r$ the latter is a point, so points of rank exactly $r$ are smooth. For $\rho<r$ it is a proper cone whose tangent space at the origin is the whole ambient space, since the symmetric matrices of rank one, respectively the alternating matrices of rank two, span that space; hence points of smaller rank are singular. Finally, the bounds zero and largest allowed rank give a point and the whole affine space of matrices.
\end{proof}

\begin{corollary}\label{geo:degree-singular}
Assume $M>0$ in a bilinear case. For each nonempty $\mathcal Y_e$, let $r$ be the rank of the matrices in $\mathcal Y_e$, and let $\mathcal D_r$ be the corresponding bounded-rank locus. Then $\overline{\mathcal Y_e}=\mathcal D_r$, and the boundary is the union of the nonempty strata of lower degree. If $r>0$ is smaller than the largest allowed rank, then the singular locus of $\overline{\mathcal Y_e}$ is, as a set, $\overline{\mathcal Y_{e'}}$, where $e'=e-m$ in type $\mathrm C$ (with $M=2m$) and $e'=e-M$ in the orthogonal cases. The closures of the strata of rank zero and of the largest rank are smooth.
\end{corollary}

\begin{proof}
Apply Propositions~\ref{geo:rank-chart} and~\ref{geo:singular} to $\mathcal Y_e=\mathcal S_r$. We have $e=mr$ in type $\mathrm C$ and $e=Mr/2$ in the orthogonal cases, so lowering the rank by one, respectively two, gives the stated values of $e'$.
\end{proof}

\subsection{\texorpdfstring{$R$-polynomials}{R-polynomials} and positivity}
\label{geo:split-count}

\begin{lemma}\label{geo:projection}
Let $G$ be a split connected reductive group over $k$, let $B\subset P$ be a standard parabolic subgroup, and let $B^-$ be the opposite Borel subgroup. For $y,z\in W^P$, the intersection $ByP/P\cap B^-zP/P$ has a finite partition into locally closed subvarieties isomorphic to
$$
 ByB/B\cap B^-vB/B,
 \qquad v\in zW_P,\quad v\le y.
$$
All maps involved are defined over $k$.
\end{lemma}

\begin{proof}
Since $y$ is the element of minimal length in its coset, the projection $\pi:G/B\to G/P$ restricts to an isomorphism
$$
 ByB/B\longrightarrow ByP/P.
$$
Moreover,
$$
 \pi^{-1}(B^-zP/P)=\coprod_{v\in zW_P}B^-vB/B.
$$
Intersecting with $ByB/B$ gives the partition, and on each piece the isomorphism is the restriction of $\pi$. The piece indexed by $v$ is empty unless $v\le y$. Since $G$ is split and $B$, $B^-$ and $P$ are defined over $k$, so are these cells, their intersections and the maps between them.
\end{proof}

We normalize the $R$-polynomials so that they count points on open Richardson varieties \cite{Deodhar}:
$$
 R_{v,y}(q)=
 \bigl|(ByB/B\cap B^-vB/B)(k)\bigr|,
 \qquad R_{v,v}(q)=1;
$$
thus, for instance, $R_{1,s}(q)=q-1$. (Other normalizations in the literature differ from this one by a sign and a power of $q$.) Let $\nu=1$ in type $\mathrm C$ and $\nu=2$ in type $\mathrm D$ (where $r$ is even), and define
$$
 \mathcal I_r=
 \{y\in W^P:y\le w_r,\ y\not\le w_{r-\nu}\},
$$
where for $r=0$ the second condition is omitted.

\begin{proposition}\label{geo:r-sum}
For the symmetric and skew-symmetric exact-rank loci,
\begin{equation}\label{geo:split-r-sum}
 |\mathcal S_r(k)|=
 \sum_{y\in\mathcal I_r}\ \sum_{\substack{v\in W_P\\v\le y}}
 R_{v,y}(q).
\end{equation}
\end{proposition}

\begin{proof}
By Proposition~\ref{geo:rank-chart},
$$
 \mathcal S_r=\bigl(X_P(w_r)\setminus X_P(w_{r-\nu})\bigr)\cap\mathcal C^-,
$$
where the second Schubert variety is omitted for $r=0$. The Bruhat cells contained in this difference are exactly the cells $ByP/P$ with $y\in\mathcal I_r$. Apply Lemma~\ref{geo:projection} with $z=1$, noting that $\mathcal C^-=B^-P/P$: the $k$-points of $ByP/P\cap\mathcal C^-$ are partitioned by the $k$-points of
$$
 ByB/B\cap B^-vB/B,\qquad v\in W_P,\quad v\le y.
$$
Summing over $y\in\mathcal I_r$ gives~\eqref{geo:split-r-sum}.
\end{proof}

To read off the coefficients, fix a reduced expression $y=s_1\cdots s_\ell$. A \emph{distinguished subexpression} is a sequence $\gamma_0=1,\gamma_1,\ldots,\gamma_\ell$ such that
$$
 \gamma_i\in\{\gamma_{i-1},\gamma_{i-1}s_i\},\qquad\text{and}\qquad
 \gamma_i=\gamma_{i-1}s_i\ \text{ whenever }\ \gamma_{i-1}s_i<\gamma_{i-1}.
$$
Put $c(\gamma)=|\{i:\gamma_i=\gamma_{i-1}\}|$ and $b(\gamma)=|\{i:\gamma_{i-1}s_i<\gamma_{i-1}\}|$. By Deodhar's theorem \cite[Theorem~1.1]{Deodhar}, the intersection indexed by $(v,y)$ is partitioned into pieces isomorphic to $\mathbb A_k^{b(\gamma)}\times(\mathbb G_{m,k})^{c(\gamma)}$, one for each distinguished subexpression $\gamma$ with $\gamma_\ell=v$. Since $G$ is split, the root subgroup coordinates defining these pieces are defined over $k$, and therefore
$$
 R_{v,y}(q)=\sum_{\gamma_\ell=v}
 q^{b(\gamma)}(q-1)^{c(\gamma)}.
$$
Let $P_r(q)$ denote the right-hand side of~\eqref{geo:split-r-sum}. With $t=q-1$ we obtain
\begin{equation}\label{geo:split-coefficient}
 [t^j]\,P_r(1+t)
 =\sum_{y\in\mathcal I_r}\
   \sum_{\substack{\gamma\text{ distinguished}\\\gamma_\ell\in W_P}}
   \binom{b(\gamma)}{j-c(\gamma)},
\end{equation}
where binomial coefficients outside the usual range are zero. In particular $P_r(q)\in\BN[q-1]$.

Now return to the radicals. For the symmetric locus of rank $r$ put $(e,h)=(mr,2m(a-r))$, where $M=2m$; for the skew-symmetric locus of rank $r$ put $(e,h)=(Mr/2,M(a-r))$. In both cases $h=aM-2e=\dim R_T$ is the rank of the bundle in Proposition~\ref{prop:parameter-varieties}, and the trace pairing identifies $\mathcal S_r$ with $\mathcal Y_e$. Counting points in the fibers gives $N_{U,e}(q)=q^hP_r(q)$. The reduced expressions and the distinguished subexpressions depend only on the Weyl group, so the index sets and the integers $b(\gamma)$, $c(\gamma)$ do not depend on $q$; the Deodhar pieces thus provide the partitions required in Section~\ref{sec:parameter-varieties}. The coefficients of $N_{U,e}(q)$ in powers of $q-1$ are given by~\eqref{geo:split-coefficient} with $b(\gamma)$ replaced by $b(\gamma)+h$.

\section{Geometric interpretation in the unitary case}\label{geo:unitary-section}

Let $\widetilde G=\GL_{2a}$ over $\overline{k}$; as in Section~\ref{geo:split-section}, this is an auxiliary group, depending only on the size $a$ of the central parameter and not on the ambient unitary group. In the basis $e_1,\ldots,e_a,f_1,\ldots,f_a$ of $E\oplus E^\vee$, let $J_0=\left[\begin{smallmatrix}0&I_a\\I_a&0\end{smallmatrix}\right]$ be the Gram matrix of a hyperbolic Hermitian form, and define
$$
 F(g)=J_0^{-1}\bigl(g^{(q)t}\bigr)^{-1}J_0,
$$
so that $\widetilde G^F\simeq\Ugrp_{2a}(q)$. Let $P$ be the stabilizer of $E$, with block sizes $(a,a)$, and let $B\subset P$ and $B^-$ be the Borel subgroups of upper and lower triangular matrices with respect to the ordered basis of Section~\ref{geo:split-section}. In that basis the Gram matrix is antidiagonal, so $B$ and $B^-$ are $F$-stable. We have $\widetilde G/P\simeq\operatorname{Gr}(a,2a)$, and on this Grassmannian $F$ sends a subspace to its orthogonal complement with respect to the $q$-sesquilinear form $(x,y)\mapsto x^{(q)t}J_0y$. For $T\in\Mat_a(\overline{k})$, viewed as a map $E\otimes_k\overline{k}\to E^\vee\otimes_k\overline{k}$, the graph $\Gamma_T=\{u+Tu:u\in E\otimes_k\overline{k}\}$ lies in the $F$-stable opposite cell $\mathcal C^-$, and
$$
 F(\Gamma_T)=\Gamma_T^{\perp}=\Gamma_{-T^{(q)t}}.
$$
The $q$-semilinear map $T\mapsto-T^{(q)t}$ therefore defines a $k$-form $\mathcal X$ of the affine space of $a\times a$ matrices, with
$$
 \mathcal X(k)=
 \{T\in\Mat_a(\overline{k}):T=-T^{(q)t}\}
   =\SHerm_a(K),
 \qquad
 \mathcal X\otimes_k\overline{k}\simeq\Mat_a(\overline{k}).
$$
Indeed, $T=-T^{(q)t}$ implies $T=T^{(q^2)}$, so the entries of $T$ lie in $K$.

\subsection{The unitary rank loci}\label{geo:unitary-descent}

Let $\widetilde W=\mathfrak S_{2a}$ be the absolute Weyl group. For $0\leq r\leq a$ put
$$
 \widetilde w_r=
 (r+1,\ldots,a,2a-r+1,\ldots,2a,
  1,\ldots,r,a+1,\ldots,2a-r)
$$
in one-line notation, and let $\mathcal F_d$ be the span of the first $d$ vectors of the ordered basis defining $B$. The first $a$ entries of $\widetilde w_r$ form the index $I_r^{\mathrm{Gr}}=(i_1,\ldots,i_a)$ of Section~\ref{geo:split-section}, so $\Lambda\in X_P(\widetilde w_r)$ if and only if $\dim(\Lambda\cap\mathcal F_{i_j})\geq j$ for $1\leq j\leq a$. We claim that
$$
 X_P(\widetilde w_r)
 =\{\Lambda\in\operatorname{Gr}(a,2a):\dim(\Lambda\cap E)\geq a-r\}.
$$
For $j\leq a-r$ we have $i_j=r+j$ and $\mathcal F_{r+j}\subseteq E=\mathcal F_a$, so $\dim(\Lambda\cap E)\geq a-r$ implies
$$
 \dim(\Lambda\cap\mathcal F_{r+j})
 \geq(a-r)+(r+j)-a=j.
$$
For $j=a-r+s$ with $1\leq s\leq r$ we have $i_j=2a-r+s$, and $\dim(\Lambda\cap\mathcal F_{i_j})\geq a+i_j-2a=j$ holds for every $\Lambda$. Conversely, when $r<a$ the condition for $j=a-r$ is exactly $\dim(\Lambda\cap E)\geq a-r$, and when $r=a$ all conditions are automatic. This proves the claim.

The diagram involution $\sigma(s_i)=s_{2a-i}$ fixes $\widetilde w_r$, so this Schubert variety is $F$-stable, and on $\mathcal C^-$ it is given by $\rank T\leq r$. The determinantal condition $\rank T\le r$ is preserved by $T\mapsto-T^{(q)t}$, so it descends to a reduced closed subvariety $\mathcal X_{\leq r}\subseteq\mathcal X$. Put
$$
 \mathcal X_r=\mathcal X_{\leq r}\setminus\mathcal X_{\leq r-1},
 \qquad \mathcal X_{\leq-1}=\varnothing.
$$
Then
$$
 \mathcal X_r(k)=\{T\in\SHerm_a(K):\rank_KT=r\},
$$
while $\mathcal X_{\leq r}\otimes_k\overline{k}$ is the ordinary variety of matrices of rank at most $r$. Writing $E_{\overline{k}}=E\otimes_k\overline{k}$, a map $E_{\overline{k}}\to E^\vee_{\overline{k}}$ of rank $r$ has a kernel $H\in\operatorname{Gr}(a-r,E_{\overline{k}})$ and induces an injection $E_{\overline{k}}/H\to E^\vee_{\overline{k}}$. These injections form an open subset of a vector bundle over $\operatorname{Gr}(a-r,E_{\overline{k}})$, so the exact-rank locus is smooth of dimension $r(a-r)+ar=r(2a-r)$ and dense in the locus of rank at most $r$. By descent, the same holds for $\mathcal X_r$.

Now consider a unitary radical with residual dimension $M>0$. By Proposition~\ref{prop:rank}, $\rank_kB_T=2M\rank_KT$ for $T\in\SHerm_a(K)$. To compare the geometric strata, choose an isomorphism $K\otimes_k\overline{k}\simeq\overline{k}\times\overline{k}$. The involution of $K$ interchanges the two factors, so a geometric skew-Hermitian parameter has the form $(A,-A^t)$. Writing $X=(X_1,X_2)$ accordingly, the radical $R_T=\{X:TX=0\}$ is given by $AX_1=0$ and $A^tX_2=0$. If $\rank A=r$, each kernel has dimension $M(a-r)$, and the commutator form over $\overline{k}$ has rank $2Mr$. Thus $\mathcal X_r$ and $\mathcal Y_{Mr}$ are defined by the same geometric rank condition; since both are reduced, the trace pairing identifies them. By Proposition~\ref{prop:parameter-varieties}, the bundle $\mathcal R_{Mr}$ and its dual have rank $d-2Mr=2M(a-r)$ over $\mathcal X_r$.

The closure of $\mathcal Y_{Mr}$ is $\mathcal X_{\leq r}$, whose base change to $\overline k$ is an ordinary determinantal variety. The Schur complement argument in the proof of Proposition~\ref{geo:singular}, applied to arbitrary matrices, shows that its singular locus is $\mathcal X_{\leq r-1}$ when $0<r<a$, while the closures of rank zero and of full rank are smooth. These properties descend to $k$, since $k$ is perfect.

To count the $F$-fixed points of these rank loci we pass to the relative Weyl group. The Frobenius acts on the simple reflections by $s_i\mapsto s_{2a-i}$, and the fixed subgroup is a Weyl group of type $\mathrm C_a$ with generators
$$
 t_i=s_i s_{2a-i}\quad(1\le i<a),\qquad t_a=s_a.
$$
We write $\iota:W(\mathrm C_a)\hookrightarrow\widetilde W$ for this embedding. For $i<a$ the two reflections in $t_i$ commute and form one $\sigma$-orbit, while the middle reflection $s_a$ is fixed. Accordingly we use the weight function
$$
 L(t_i)=2\quad(i<a),\qquad L(t_a)=1,
$$
which corresponds to the parameters $q^2$ and $q$ of the BN-pair of $\widetilde G^F\simeq\Ugrp_{2a}(q)$. For a reduced expression $w=t_{i_1}\cdots t_{i_\ell}$ put $\ell_L(w)=\sum_jL(t_{i_j})$; this does not depend on the expression. Let $W_{P,F}=\langle t_1,\ldots,t_{a-1}\rangle\simeq\mathfrak S_a$ be the relative Weyl group of the Levi subgroup, which we identify with $\mathfrak S_a$, and write $W(\mathrm C_a)^{\mathfrak S_a}$ for the set of minimal length representatives of $W(\mathrm C_a)/\mathfrak S_a$. For the element $w_r$ of type $\mathrm C$ defined in~\eqref{geo:wr} we have $\iota(w_r)=\widetilde w_r$ and
$$
 \ell_L(w_r)=\ell_{\widetilde W}(\widetilde w_r)=r(2a-r).
$$
The first equality holds because a reduced word in the $t_i$ expands to a reduced word in $\widetilde W$ in which $t_i$ contributes $L(t_i)$ simple reflections \cite[Lemma~5.3]{Hultman}. For the second, count inversions of $\widetilde w_r$: each of the entries $r+1,\ldots,a$ exceeds the $r$ entries $1,\ldots,r$ to its right, giving $r(a-r)$ inversions, and each of the entries $2a-r+1,\ldots,2a$ exceeds all $a$ entries to its right, giving $ar$ more.

\begin{lemma}\label{geo:folding}
Let $\sigma(s_i)=s_{2a-i}$ act on $\widetilde W$, and let $\widetilde W_P=\mathfrak S_a\times\mathfrak S_a$. Write $\widetilde W^{\widetilde W_P}$ for the set of minimal length representatives of $\widetilde W/\widetilde W_P$, and for a $\sigma$-stable subset $A\subseteq\widetilde W$ put $A^\sigma=\{w\in A:\sigma(w)=w\}$. Then $\iota$ maps $W(\mathrm C_a)$ onto $\widetilde W^\sigma$, and
$$
 u\le v\ \text{in }W(\mathrm C_a)
 \quad\Longleftrightarrow\quad
 \iota(u)\le\iota(v)\ \text{in }\mathfrak S_{2a}.
$$
Furthermore,
$$
 (\widetilde W_P)^\sigma=\iota(\mathfrak S_a),\qquad
 (\widetilde W^{\widetilde W_P})^\sigma
   =\iota\bigl(W(\mathrm C_a)^{\mathfrak S_a}\bigr).
$$
\end{lemma}

\begin{proof}
The description of the fixed subgroup and the compatibility of the Bruhat orders are \cite[Theorem~5.1, Remark~5.2 and Theorem~5.5]{Hultman}; applied to the $\sigma$-stable parabolic subgroup $\widetilde W_P$, the same description gives $(\widetilde W_P)^\sigma=\iota(\mathfrak S_a)$. Now let $x=\iota(u)$ be a fixed element. For $i<a$, the commuting reflections $s_i$ and $s_{2a-i}$ are either both right descents of $x$ or both right ascents, because $\sigma(x)=x$. Hence $u$ has a right descent at $t_i$ exactly when $x$ has right descents at both $s_i$ and $s_{2a-i}$. Since minimal coset representatives are characterized by the absence of right descents in the Levi subgroup, the last assertion follows.
\end{proof}

\subsection{Unequal-parameter \texorpdfstring{$R$-polynomials}{R-polynomials}}\label{geo:unitary-r}

Let $R^L_{v,y}$, for $v,y\in W(\mathrm C_a)$, be the unequal-parameter $R$-polynomials, normalized to count points. They are determined by $R^L_{v,v}=1$, by $R^L_{v,y}=0$ for $v\not\le y$, and by the recursion
\begin{equation}\label{geo:unequal-recurrence}
 R^L_{v,y}=
 \begin{cases}
 R^L_{vs,ys},&vs<v,\\
 q^{L(s)}R^L_{vs,ys}+(q^{L(s)}-1)R^L_{v,ys},&vs>v,
 \end{cases}
\end{equation}
valid for every simple reflection $s$ with $ys<y$. This normalization corresponds to the quadratic relation $(T_s-q^{L(s)})(T_s+1)=0$, after rescaling the generators in \cite[Section~3.2 and Lemma~4.4]{Lusztig}; other conventions differ by a sign and a power of $q$.

The next lemma is the unitary analogue of the familiar point count for open Richardson varieties. The unequal parameters appear because the relative root subgroups have orders $q^2$ and $q$.

\begin{lemma}\label{geo:unitary-r-points-lemma}
For $v,y\in W(\mathrm C_a)$,
\begin{equation}\label{geo:unitary-r-points}
 R^L_{v,y}(q)=
 \bigl|(B\iota(y)B/B\cap B^-\iota(v)B/B)^F\bigr|.
\end{equation}
\end{lemma}

\begin{proof}
Let $N_{v,y}$ denote the right-hand side of~\eqref{geo:unitary-r-points}. Since $B$ is connected and $F$-stable, Lang's theorem gives $(\widetilde G/B)^F\simeq\widetilde G^F/B^F$; we call these cosets \emph{rational chambers}. For a relative simple reflection $s$, let $P_s^F$ be the corresponding minimal parabolic subgroup; an \emph{$s$-panel} is a fiber of the map $\widetilde G^F/B^F\to\widetilde G^F/P_s^F$. By the rational Bruhat decomposition \cite[Theorem~24.1]{MalleTesterman}, the fixed points of the absolute cell indexed by $\iota(y)$ form the relative cell indexed by $y$, and similarly for the opposite cells. The relative simple reflections of $\widetilde G^F$ are $t_1,\ldots,t_a$, and the corresponding relative root subgroups have orders $q^2,\ldots,q^2,q$. Put $q_s=q^{L(s)}$. Then $|P_s^F/B^F|=q_s+1$, so each $s$-panel contains $q_s+1$ chambers; see \cite[Sections~23.2 and~24.2]{MalleTesterman}.

Let $w_0$ be the longest element of $W(\mathrm C_a)$, and choose an $F$-fixed representative $\dot w_0$ of $\iota(w_0)$ with $B^-=\dot w_0B\dot w_0^{-1}$. Put $c_+=B/B$ and $c_-=\dot w_0B/B$. We call the Weyl distance from $c_+$ the \emph{index} of a chamber. If $c\in(B^-\iota(u)B/B)^F$, we call $u$ the \emph{opposite index} of $c$; its Weyl distance from $c_-$ is $w_0u$. Within an $s$-panel, the Weyl distances from a fixed chamber take two values, which differ by right multiplication by $s$, and the rank-one Bruhat decomposition of $P_s^F$ shows that one chamber of the panel is at the shorter distance and $q_s$ chambers are at the longer one. Since $\ell(w_0u)=\ell(w_0)-\ell(u)$, it follows that in every $s$-panel the larger of the two opposite indices occurs once and the smaller one $q_s$ times.

Suppose that $ys<y$. Each chamber of index $y$ lies in a unique $s$-panel, and, applying the panel rule with respect to $c_+$, this panel consists of one chamber $c_0$ of index $ys$ and $q_s$ chambers of index $y$. Conversely, each chamber of index $ys$ determines such a panel. We now count the chambers of opposite index $v$ in these panels.

If $vs<v$, then a chamber $c_0$ of opposite index $vs$ contributes exactly one chamber of opposite index $v$, while a chamber $c_0$ of opposite index $v$ contributes none. Hence
$$
 N_{v,y}=N_{vs,ys}\qquad\text{if }vs<v.
$$
If $vs>v$, a chamber $c_0$ of opposite index $vs$ contributes $q_s$ chambers of opposite index $v$, and a chamber $c_0$ of opposite index $v$ contributes $q_s-1$ further ones. Thus
$$
 N_{v,y}=q^{L(s)}N_{vs,ys}
 +(q^{L(s)}-1)N_{v,ys}\qquad\text{if }vs>v.
$$
For $v=y$ the intersection is the single point $\dot vB/B$, where $\dot v$ is an $F$-fixed representative of $\iota(v)$, so $N_{v,v}=1$. If $v\not\leq y$, the intersection is empty by Lemma~\ref{geo:folding} and the nonemptiness criterion for Richardson varieties. These are the recursion and the initial conditions~\eqref{geo:unequal-recurrence}, and induction on $\ell(y)$ gives $N_{v,y}=R^L_{v,y}(q)$.
\end{proof}

\begin{theorem}\label{geo:unitary-count}
In $W(\mathrm C_a)$ put
$$
 \mathcal I_r^{\mathrm U}=
 \{y\in W(\mathrm C_a)^{\mathfrak S_a}:y\le w_r,\ y\not\le w_{r-1}\},
$$
where for $r=0$ the second condition is omitted. Then the number of skew-Hermitian $a\times a$ matrices over $K$ of rank $r$ is
\begin{equation}\label{geo:unitary-r-sum}
 |\mathcal X_r(k)|=
 \sum_{y\in\mathcal I_r^{\mathrm U}}\
 \sum_{\substack{v\in\mathfrak S_a\\v\le y}}R^L_{v,y}(q),
 \qquad L=(2,\ldots,2,1).
\end{equation}
\end{theorem}

\begin{proof}
The rational points of a $k$-variety are the fixed points of its Frobenius, so
$$
 \mathcal X_r(k)=
 \bigl[
 \bigl(X_P(\widetilde w_r)\setminus X_P(\widetilde w_{r-1})\bigr)
 \cap\mathcal C^-
 \bigr]^F,
$$
where the second Schubert variety is omitted for $r=0$. The Frobenius maps $B\widetilde yP/P$ onto $B\sigma(\widetilde y)P/P$. Since distinct Bruhat cells are disjoint, a cell containing an $F$-fixed point has a $\sigma$-fixed index. The cells in the rank-$r$ locus are those with $\widetilde y\leq\widetilde w_r$ and $\widetilde y\not\leq\widetilde w_{r-1}$, and since $\widetilde w_j=\iota(w_j)$, Lemma~\ref{geo:folding} identifies the $\sigma$-fixed indices among them with the elements $\iota(y)$, $y\in\mathcal I_r^{\mathrm U}$.

Fix such a $y$. Since $\iota(y)$ is a minimal coset representative by Lemma~\ref{geo:folding}, the projection $\pi:\widetilde G/B\to\widetilde G/P$ restricts to an isomorphism
$$
 B\iota(y)B/B\longrightarrow B\iota(y)P/P,
$$
and this isomorphism commutes with $F$ because $B$ and $P$ are $F$-stable. As in Lemma~\ref{geo:projection},
$$
 \pi^{-1}(\mathcal C^-)
 =\coprod_{\widetilde v\in\widetilde W_P}B^-\widetilde vB/B.
$$
Intersect this partition with $B\iota(y)B/B$. The Frobenius maps the piece indexed by $\widetilde v$ onto the piece indexed by $\sigma(\widetilde v)$, and since the pieces are disjoint, a fixed point can only lie in a piece with $\sigma(\widetilde v)=\widetilde v$. By Lemma~\ref{geo:folding} these indices are the elements $\iota(v)$ with $v\in\mathfrak S_a$, and the corresponding intersection is nonempty exactly when $v\le y$. By Lemma~\ref{geo:unitary-r-points-lemma} it has $R^L_{v,y}(q)$ fixed points. Summing over $v$ and then over $y$ gives~\eqref{geo:unitary-r-sum}.
\end{proof}

As in the bilinear case, the coefficients can be read off from distinguished subexpressions. Fix a reduced expression $y=\tau_1\cdots\tau_\ell$ with $\tau_i\in\{t_1,\ldots,t_a\}$, use the distinguished subexpressions of Section~\ref{geo:split-section} for this word, and put
$$
 b_L(\gamma)=
 \sum_{i:\ \gamma_{i-1}\tau_i<\gamma_{i-1}}L(\tau_i),
 \qquad
 a_j(\gamma)=
 \bigl|\{i:\gamma_i=\gamma_{i-1},\ L(\tau_i)=j\}\bigr|
 \quad(j=1,2).
$$
Applying~\eqref{geo:unequal-recurrence} successively along the word, a forced downward step contributes $q^{L(\tau_i)}$, a skipped step contributes $q^{L(\tau_i)}-1$, and an upward step contributes $1$. Hence
$$
 R^L_{v,y}(q)=\sum_{\gamma_\ell=v}
 q^{b_L(\gamma)}(q-1)^{a_1(\gamma)}(q^2-1)^{a_2(\gamma)}.
$$
With $t=q-1$, each summand becomes
$$
 q^{b_L}(q-1)^{a_1}(q^2-1)^{a_2}
 =(1+t)^{b_L}\,t^{a_1+a_2}(t+2)^{a_2},
$$
and its contribution to the coefficient of $t^j$ is
\begin{equation}\label{geo:weighted-coefficient}
 \sum_{\substack{0\le p\le b_L,\ 0\le u\le a_2\\p+u=j-a_1-a_2}}
 \binom{b_L}{p}\binom{a_2}{u}2^{a_2-u}.
\end{equation}
Summing over $y\in\mathcal I_r^{\mathrm U}$ and over the distinguished subexpressions ending in $\mathfrak S_a$ gives the coefficients of~\eqref{geo:unitary-r-sum}. The Weyl group, the weights and the subexpressions do not depend on $q$, so~\eqref{geo:unitary-r-sum} is a single polynomial in $\BN[q-1]$.

\begin{remark}\label{geo:parity}
For the ambient group $\Ugrp_{2n+1}(q)$ the last relative parameter is $q^3$; see \cite[Sections~23.2 and~24.2]{MalleTesterman}. This plays no role here: the auxiliary group $\widetilde G^F\simeq\Ugrp_{2a}(q)$ has parameters $q^2,\ldots,q^2,q$ regardless of the parity of the ambient group, and these are the parameters in~\eqref{geo:unitary-r-sum}.
\end{remark}

Finally, since the character parameter bundle has rank $2M(a-r)$ over $\mathcal X_r$, we obtain for $e=Mr$
$$
 N_{U,Mr}(q)
 =q^{2M(a-r)}|\mathcal X_r(k)|
 =q^{2M(a-r)}
 \sum_{y\in\mathcal I_r^{\mathrm U}}
 \sum_{\substack{v\in\mathfrak S_a\\v\le y}}
 R^L_{v,y}(q).
$$
Its coefficients in powers of $q-1$ are given by~\eqref{geo:weighted-coefficient} with $b_L$ replaced by $b_L+2M(a-r)$. Together with Section~\ref{geo:split-section}, this completes the proof of Theorem~\ref{thm:intro-geometry}.

\begin{remark}\label{geo:extensions}
The formulas of this section concern the unitary family as $q$ varies. They do not give the number of points of the fixed $k$-variety $\mathcal X_r$ over the finite extensions of $k$ by substituting $q^m$ for $q$, because the unitary form splits over extensions of even degree. For instance, let $a=2$, $r=1$ and $Q=q^m$. If $m$ is odd, then $F^m(T)=-T^{(Q)t}$, so $\mathcal X_1(\BF_{q^m})$ is the set of skew-Hermitian $2\times2$ matrices of rank one over $\BF_{Q^2}$ and
$$
 |\mathcal X_1(\BF_{q^m})|=(Q-1)(Q^2+1).
$$
If $m$ is even, then $F^m(T)=T^{(Q)}$, so $\mathcal X_1(\BF_{q^m})$ is the set of all $2\times2$ matrices of rank one over $\BF_Q$ and
$$
 |\mathcal X_1(\BF_{q^m})|=(Q-1)(Q+1)^2.
$$
A partition of $\mathcal X_1$ into $k$-subvarieties isomorphic to $\mathbb A_k^{b}\times(\mathbb G_{m,k})^{c}$ would give $|\mathcal X_1(\BF_{q^m})|=P(q^m)$ for a single polynomial $P$ and all $m$. The odd values of $m$ force $P(t)=(t-1)(t^2+1)$, which contradicts the count for even $m$. So no such partition exists, and this is why we counted the rational points of the unitary rank loci directly instead of using the argument at the end of Section~\ref{sec:parameter-varieties}.
\end{remark}

\section*{Declaration of generative AI and AI-assisted technologies in the manuscript preparation process}

During manuscript preparation, we used ChatGPT (OpenAI) to improve English expression and readability, assist with editing and formatting mathematical formulas, provide coding assistance, and help locate relevant literature. 
We reviewed and revised the resulting text and formulas, checked the relevant code, and verified the bibliographic information and cited claims against the original sources. 
The authors retain full responsibility for the accuracy, originality, and integrity of the final manuscript.

\end{document}